\documentclass{article}%
\usepackage{amsmath}
\usepackage{amsfonts}
\usepackage{amssymb}
\usepackage{graphicx}
\usepackage{xcolor}%
\providecommand{\U}[1]{\protect\rule{.1in}{.1in}}
\newtheorem{theorem}{Theorem}[section]

\newtheorem{conjecture}[theorem]{Conjecture}
\newtheorem{corollary}[theorem]{Corollary}

\newtheorem{definition}[theorem]{Definition}
\newtheorem{example}[theorem]{Example}

\newtheorem{lemma}[theorem]{Lemma}

\newtheorem{proposition}[theorem]{Proposition}
\newtheorem{remark}[theorem]{Remark}

\newtheorem{question}[theorem]{Question}

\newenvironment{proof}[1][Proof]{\noindent\textbf{#1.} }{\ \rule{0.5em}{0.5em}}
\date{}
\begin{document}

\title{Almost invariant subspaces of shift operators and products of Toeplitz and
Hankel operators}
\author{Caixing Gu, In Sung Hwang, Hyoung Joon Kim, Woo Young Lee, and Jaehui Park}
\maketitle

\begin{abstract}
In this paper we formulate the almost invariant subspaces theorems
of backward shift operators in terms of the ranges or kernels of
product of Toeplitz and Hankel operators. This approach simplifies
and gives more explicit forms of these almost invariant subspaces
which are derived from \ related nearly backward shift invariant
subspaces with finite defect. Furthermore, this approach also leads
to the surprising result that  the almost invariant subspaces of
backward shift operators are the same as the almost invariant
subspaces of forward shift operators which were treated only briefly
in literature.

\


\bigskip


\end{abstract}


\section{Introduction}

The recent study of almost invariant subspaces inside Banach spaces
attributes its starting point {\color{red}to} the paper by
Androulakis-Popov-Tcaciuc-Troitsky \cite{APTT} in 2009, see also
\cite{PopovT} \cite{Sirotkin} \cite{Sirotkin2} \cite{Tacacius}
\cite{Jung}. But an equivalent definition on Hilbert spaces occurred
in Hoffman \cite{Hoffman} in 1978 in connection with essentially
invariant subspaces \cite{BrownPearcy} in 1971 (see Definition
\ref{def2} and Definition \ref{defBP}).  In the sequel, let $H$ be a
complex Hilbert space and $B(H)$ be the algebra of bounded linear
operators acting on $H$. For an operator $T\in B(H)$, write $R(T)$
and $N(T)$ for the range and the kernel of $T$, respectively.
If $M$ is a closed subspace of $H$, then the orthogonal
projection from $H$ onto $M$ is denoted by $P_M$.

\begin{definition}
\label{def1} \textrm{\cite{APTT}} Let $T\in B(H).$ A closed subspace
$H_{1}$ of $H$ is called an almost invariant subspace of $T$ if
there is a finite dimensional subspace $W$ such that $TH_{1}\subset
H_{1}+W.$ The minimal possible dimension $n$ of $W$ is called the
defect of $H_{1},$ denoted by
$\varsigma(H_{1})=\varsigma(T,H_{1})=n,$ and $W$ is called a minimal
defect space of $H_{1}.$ If $W$ is a minimal defect space of $H_{1}$
such that $W\perp H_{1}$, then $W$ is called a minimal orthogonal
defect space of $H_{1}.$
\end{definition}

We say $H_{1}$ is an almost reducing subspace of $T$ if both $H_{1}$
and $H_{1}^{\perp}$ are almost invariant for $T$ or equivalently
$H_{1}$ is almost invariant for both $T$ and $T^{\ast}.$ To avoid
triviality, often $H_{1}$ is assumed such that both $H_{1}$ and
$H_{1}^{\perp}$ are infinite dimensional, such an $H_{1}$ is called
a  {\it half-space}. The study of almost invariant
subspaces is motivated by the Invariant Subspace Problem which is
still open on Hilbert spaces. By a series of papers from
several authors \cite{PopovT} \cite{Sirotkin} \cite{Sirotkin2} \cite{Tacacius},
Tcaciuc proved the following important theorem.

\begin{theorem}
Let $X$ be a separable Banach space and $T$ be a bounded linear operator on
$X.$ Then $T$ has an almost invariant half-space with defect at most $1.$
\end{theorem}

On the other hand, the celebrated Beurling-Lax-Halmos (BLH) invariant subspace
Theorem \cite{Beurling} \cite{Lax} \cite{Halmos} for the (forward) shift
operator is a cornerstone of modern analysis. It has played an important role
in operator theory, function theory and applications. The following question
is natural.

\begin{question}
\label{question}What are the almost invariant subspaces of the shift operator?
\end{question}

It turns out the almost invariant subspaces of the backward shift operator are
characterized recently, see \cite{ChalendarCP} \cite{ChalendarGP} \cite{ChDas}
\cite{OLoughlin}. By formulating the results of these papers using the ranges
of Toeplitz and Hankel operators, we will answer the above motivating
question. Now we introduce the function theoretic background of the shift
operator. Let $L^{2}$ be the space of square integrable functions on the unit
circle $\mathbb{T}$ with respect to the normalized Lebesgue measure. Let
$H^{2}$ be the Hardy space on the open unit disk $\mathbb{D}$.
$L^{\infty
}$ and $H^{\infty}$ are the algebras of bounded functions in $L^{2}$ and
$H^{2}$ respectively. Let $E$ and $F$ be two complex separable Hilbert spaces.
Let $B(E,F)$ be the set of bounded linear operators from $E$ into
$F.$ $L_{E}^{2}$ and $H_{E}^{2}$ denote $E$-valued $L^{2}$ and $H^{2}$
spaces, respectively. $L_{B(E)}^{\infty}$ and $H_{B(E)}%
^{\infty}$ are operator-valued $L^{\infty}$ and $H^{\infty}$ algebras, and
$L_{B(E,F)}^{\infty}$ and $H_{B(E,F)}^{\infty}$ are operator-valued
$L^{\infty}$ and $H^{\infty}$ spaces.

Denote by $P$ the projection from $L_{E}^{2}$ to $H_{E}^{2}$ and $Q:=I-P$ the
projection from $L_{E}^{2}$ to $\overline{zH_{E}^{2}}:=L_{E}^{2}\ominus
H_{E}^{2},$ where $E$ is any complex separable Hilbert space. Let $\Phi\in
L_{B(E,F)}^{\infty}.$ The multiplication operator by $\Phi$ from $L_{E}^{2}$
into $L_{F}^{2}$ is denoted by $M_{\Phi}.$ The (block) Toeplitz operator
$T_{\Phi}$ from $H_{E}^{2}$ into $H_{F}^{2}$ is defined by%
\[
T_{\Phi}h=P\left[  \Phi h\right]  ,h\in H_{E}^{2}.
\]
When $E=F,$ let $T_{z}$($=T_{zI_{E}}$) denote the shift operator on $H_{E}%
^{2}$ for some $E$ which the context will make clear. We will also use
$S_{E}$ or just $S$ to denote this shift operator $T_{z}$ on $H_{E}^{2}$ and
$S_{E}^{\ast}$ or just $S^{\ast}$ for the backward shift.
We note that $E$ may be regarded as a subspace
$H^2_E$ in the sense that if for each $x \in E$, we define
$$
f_x(z):=x \ \ \hbox{for all} \ z \in \mathbb T,
$$
then $f_x\in H^2_E$, so that $E\cong \{f_x: x \in E\}\subseteq
H^2_E$. Thus if there is no confusion in the context of $H^2_E$, we
will still use the notation $P_E$. Observe that
\[
I-S_{E}S_{E}^{\ast}=P_{E}.
\]
The
Toeplitz operator $T_{\Phi}=A$ is characterized by the operator equation
$S_{F}^{\ast}AS_{E}=A.$

Let $J$ be defined on $L_{E}^{2}$ by
\[
Jf(z)=\overline{z}f(\overline{z}),\ \ f\in L_{E}^{2}.
\]
$J$ maps $\overline{zH_{E}^{2}}$ onto $H_{E}^{2}$, and $J$ maps $H_{E}%
^{2}$ onto $\overline{zH_{E}^{2}}.$ Furthermore $J$ is a unitary operator,
\[
J^{\ast}=J,J^{2}=I,\text{ }JQ=PJ,\text{ and }JP=QJ.
\]
For $\Phi\in L^{\infty}$, the Hankel operator $H_{\Phi}$ from
$H_{E}^{2}$ into $H_{F}^{2}$ is defined by%
\[
H_{\Phi}h=JQ\left[  \Phi h\right]  =PJ\left[  \Phi h\right]
,\ \ h\in H_{E}^{2}.
\]
The Hankel operator $H_{\Phi}=A$ is characterized by the operator equation
$AS_{E}=S_{F}^{\ast}A$. Set $\widetilde{\Phi}(z)=\Phi(\overline{z})^{\ast}.$
It is easy to check that $H_{\Phi}^{\ast}=H_{\widetilde{\Phi}}.$ The Toeplitz
and Hankel operators are connected in the following basic formula:%
\begin{equation}
T_{\Omega\Psi}-T_{\Omega}T_{\Psi}=H_{\Omega^{\ast}}^{\ast}H_{\Psi} \label{thc}%
\end{equation}
for symbols $\Omega$ and $\Psi$ with compatible dimensions, see \cite{BS2}
\cite{Pe}.

We emphasize that again Toeplitz and Hankel operators $T_{\Phi}$ and $H_{\Phi}$
could be between different spaces $H_{E}^{2}$ and $H_{F}^{2}$ according to
whether the symbol $\Phi\ $acts between different spaces $E$ and $F.$ In this
paper, all underlying Hilbert spaces $E,$ $F,$ and $E_{i}$ for the
vector-valued Hardy spaces $H_{E}^{2},$ $H_{F}^{2},$ and $H_{E_{i}}^{2}$ will
be finite dimensional. We also note that $T_{\Phi}$ and $H_{\Phi}$ could be
unbounded operators, where $\Phi\in L_{B(E,F)}^{2}$, that is, $\Phi$ is a
matrix-valued function whose entries belong to $L^{2}.$ In this case an
operator equation involving $T_{\Phi}$ and $H_{\Phi}$ is valid when it acts on
polynomials or $H^{\infty}$ functions, see for example \cite{GuH} where the
product of unbounded operator $T_{\Phi}H_{\Psi}$ can be interpreted by
using the
bilinear form $\left\langle T_{\Phi}H_{\Psi}h,g\right\rangle :=\left\langle
H_{\Psi}h,T_{\Phi^{\ast}}g\right\rangle $ for $h,g$ being polynomials or
$H^{\infty}$ functions.

The $S_{E}^{\ast}$-almost invariant subspaces (and related nearly $S_{E}%
^{\ast}$-invariant subspaces) have been characterized by
Chalendar-Chevrot-Partington \cite{ChalendarCP}, Chalendar-Gallardo-Partington
\cite{ChalendarGP}, Chattopadhyay-Das-Pradhan \cite{ChDas}, and independently
in O'Loughlin \cite{OLoughlin}, see Theorem \ref{nearlym} and Corollary
\ref{almostm} below for their developments and characterizations.

In this paper, we reformulate their results using operator ranges
and prove a number of results on the almost invariant subspaces of
$S_{E}^{\ast}$ and $S_{E}.$ Our approach achieve three goals: first the
reformulation simplifies the results of $S_{E}^{\ast}$-almost invariant
subspaces in \cite{ChalendarCP} \cite{ChalendarGP} \cite{ChDas}
\cite{OLoughlin} which were derived as consequences of related nearly
$S_{E}^{\ast}$-invariant subspaces, the reformulation bypasses the use of
nearly $S_{E}^{\ast}$-invariant subspaces; second our approach leads to more
explicit examples of $S_{E}^{\ast}$-invariant subspaces; third surprisingly,
we prove that almost invariant subspaces of $S_{E}$ are the same as the almost
invariant subspaces of $S_{E}^{\ast}$ and thus answer our motivating Question
\ref{question}\ completely.

Let us briefly outline the plan of the paper. In Section 2, we show that the
closures of ranges of finite sums of finite products of Hankel and Toeplitz
operators under a mild condition are $S_{E}^{\ast}$-almost
invariant and $S_{E}$-almost invariant. In Section 3, we
reformulate their results (see Theorem \ref{nearlym}) as follows.

\begin{theorem}
\label{thma}A closed subspace $M$ of $H_{F}^{2}$ is $S_{F}^{\ast}$-almost
invariant if and only if one of the following holds:

\begin{itemize}
\item[\textrm{(1)}] $M=R(T_{\Theta})$, where $\Theta\in H_{B(E,F)}^{\infty}$
is inner.

\item[\textrm{(2)}] $M=R(T_{\Phi}\left(  I_{E_{1}}-T_{\Theta}T_{\Theta}^{\ast
}\right)  )$, where $\Theta\in H_{B(E,E_{1})}^{\infty}$ is inner and pure, and
$\Phi\in H_{B(E_{1},F)}^{2}$ is such that $T_{\Phi}\left(  I_{E_{1}}-T_{\Theta
}T_{\Theta}^{\ast}\right)  $ is a partial isometry.
\end{itemize}
\end{theorem}

The idea of using ranges of partial isometries to represent almost
invariant subspaces is inspired by \cite{CRoss} \cite{Timotin} \cite{GuLuo},
where the invariant subspaces of $S_{F}\oplus S_{F}^{\ast}$ are represented as
ranges of partial isometries involving Toeplitz and Hankel operators. Such an
approach also enables us to observe the following.

\begin{corollary}
Let $M$ be a closed subspace of $H_{F}^{2}.$ Then $M$ is $S_{F}^{\ast}$-almost
invariant if and only if $M$ is $S_{F}$-almost invariant. Consequently, If $M$
is $S_{F}$-almost invariant, then $M$ is $S_{F}$-almost reducing.
\end{corollary}

The above corollary gives explicit forms of $S_{F}$-almost invariant
subspaces, otherwise, $S_{F}$-almost invariant subspaces are represented as
$M^{\perp}$, where $M$ is as in Theorem \ref{thma}. See Remark 3.3 in
\cite{ChalendarGP} and Remark 3.7 in \cite{ChDas}, where it is observed that
$S_{E}^{\ast}M\subset M\oplus W$ if and only if $S_{E}\left(  M\oplus
W\right)  ^{\perp}\subset\left(  M\oplus W\right)  ^{\perp}\oplus W.$ The
orthogonal complement $M^{\perp}$ seems difficult to identify, for example, we
know $M$ is infinite dimensional if and only if $\Theta$ is not a finite
Blaschke-Potapov product, but we know $M^{\perp}$ to be infinite dimensional
in some special cases. In Section 4, we study in more details $M=R(T_{\Phi
}\left(  I_{E_{1}}-T_{\Theta}T_{\Theta}^{\ast}\right)  )$, where $\Phi$ is
also inner and give more explicit examples of $S_{E}^{\ast}$-invariant or
$S_{E}$-invariant subspaces.

\

An observation on Banach spaces \cite{APTT} says that a closed subspace $M$ is
$T$-almost invariant if and only if $M$ is $\left(  T+T_{0}\right)
$-invariant for some finite rank operator $T_{0},$ see Lemma \ref{apttl}
below. In section 5, we first give a more precise result on Hilbert spaces by
using the equivalent definition of almost invariant subspaces in
\cite{Hoffman}. Namely, if $M$ is $T$-almost invariant, we identify all finite
rank operators $T_{0}$ such that $M$ is $\left(  T+T_{0}\right)  $-invariant.
We then apply this theorem to $S_{F}^{\ast}$-almost invariant subspaces. For a
given $M=R(T_{\Phi}\left(  I_{E_{1}}-T_{\Theta}T_{\Theta}^{\ast}\right)  )$ as
in Theorem \ref{thma}, we write down all finite rank operators $T_{0}$ in
terms of $\Phi,$ $\Theta,$ $M$ such that $M$ is $\left(  S_{F}^{\ast}%
+T_{0}\right)  $-invariant or $\left(  S_{F}+T_{0}\right)  $-invariant or
$\left(  S_{F}+T_{0}\right)  $-reducing.


\section{The ranges of products of Toeplitz and Hankel operators are $S^{\ast
}$-almost invariant}

\

We begin with:

\begin{lemma}
\label{basic}Let $\Psi\in L_{B(E,E_{1})}^{\infty}$ and $\Phi\in L_{B(E_{1}%
,F)}^{\infty}.$ Then $S_{F}^{\ast}T_{\Phi}H_{\Psi}-T_{\Phi}H_{\Psi}S_{E}$ is
of finite rank, more precisely,%
\begin{equation}
S_{F}^{\ast}T_{\Phi}H_{\Psi}-T_{\Phi}H_{\Psi}S_{E}=S_{F}^{\ast}T_{\Phi
}P_{E_{1}}H_{\Psi}. \label{one}%
\end{equation}
Also, we have
\begin{equation}
S_{F}H_{\Phi}T_{\Psi}-H_{\Phi}T_{\Psi}S_{E}^{\ast}=H_{\Phi}S_{E_{1}}^{\ast
}T_{\Psi}P_{E}-P_{F}H_{\Phi}S_{E_{1}}^{\ast}T_{\Psi}+S_{F}H_{\Phi}P_{E_{1}%
}T_{\Psi}. \label{two}%
\end{equation}
\end{lemma}

\begin{proof}
Recall $S_{F}^{\ast}T_{\Phi}S_{E_{1}}=T_{\Phi}$ and $S_{E_{1}}^{\ast}H_{\Psi
}=H_{\Psi}S_{E}.$ Then%
\begin{align}
S_{F}^{\ast}T_{\Phi}S_{E_{1}}S_{E_{1}}^{\ast}  &  =T_{\Phi
}S_{E_{1}}^{\ast},\nonumber\\
S_{F}^{\ast}T_{\Phi}\left(  I-P_{E_{1}}\right)   &  =T_{\Phi}S_{E_{1}}^{\ast
}\ \ \text{and hence},\ S_{F}^{\ast}T_{\Phi}=T_{\Phi}S_{E_{1}%
}^{\ast}+S_{F}^{\ast}T_{\Phi}P_{E_{1}} \label{ts}%
\end{align}
and%
\begin{align*}
S_{F}^{\ast}T_{\Phi}H_{\Psi}  &  =\left(  T_{\Phi}S_{E_{1}}^{\ast}+S_{F}%
^{\ast}T_{\Phi}P_{E_{1}}\right)  H_{\Psi}\\
&  =T_{\Phi}S_{E_{1}}^{\ast}H_{\Psi}+S_{F}^{\ast}T_{\Phi}P_{E_{1}}H_{\Psi}\\
&  =T_{\Phi}H_{\Psi}S_{E}+S_{F}^{\ast}T_{\Phi}P_{E_{1}}H_{\Psi}.
\end{align*}
This proves (\ref{one}).

Similarly, it follows from $S_{F}^{\ast}H_{\Phi}=H_{\Phi}S_{E_{1}}$ that
\begin{align}
S_{F}S_{F}^{\ast}H_{\Phi}S_{E_{1}}^{\ast}  &  =S_{F}H_{\Phi}S_{E_{1}}S_{E_{1}%
}^{\ast},\nonumber\\
\left(  I-P_{F}\right)  H_{\Phi}S_{E_{1}}^{\ast}  &  =S_{F}H_{\Phi}\left(
I-P_{E_{1}}\right)  ,\nonumber\\
S_{F}H_{\Phi}  &  =H_{\Phi}S_{E_{1}}^{\ast}-P_{F}H_{\Phi}S_{E_{1}}^{\ast
}+S_{F}H_{\Phi}P_{E_{1}} \label{hs}%
\end{align}
and%
\begin{align*}
S_{F}H_{\Phi}T_{\Psi}  &  =\left(  H_{\Phi}S_{E_{1}}^{\ast}-P_{F}H_{\Phi
}S_{E_{1}}^{\ast}+S_{F}H_{\Phi}P_{E_{1}}\right)  T_{\Psi}\\
&  =H_{\Phi}S_{E_{1}}^{\ast}T_{\Psi}-P_{F}H_{\Phi}S_{E_{1}}^{\ast}T_{\Psi
}+S_{F}H_{\Phi}P_{E_{1}}T_{\Psi}\\
&  =H_{\Phi}T_{\Psi}S_{E}^{\ast}+H_{\Phi}S_{E_{1}}^{\ast}T_{\Psi}P_{E}%
-P_{F}H_{\Phi}S_{E_{1}}^{\ast}T_{\Psi}+S_{F}H_{\Phi}P_{E_{1}}T_{\Psi}.
\end{align*}
The proof is complete.
\end{proof}

\

Let $H_{1}$ be an almost invariant subspace of $T\ $and $W$ is the minimal
orthogonal defect space of $H_{1}.$ So $TH_{1}\subset H_{1}\oplus W.$ Since
$TH_{1}\subset H_{1}\oplus W$ is equivalent to $T^{\ast}\left(  H_{1}\oplus
W\right)  ^{\perp}\subset H_{1}^{\perp}=\left(  H_{1}\oplus W\right)  ^{\perp
}\oplus W.$ Thus $H_{1}$ is $T$-almost invariant with defect $n$ if and only
if $\left(  H_{1}\oplus W\right)  ^{\perp}$ is $T^{\ast}$-almost invariant
with defect $n.$ Furthermore, $T^{\ast}H_{1}^{\perp}=T^{\ast}\left[  \left(
H_{1}\oplus W\right)  ^{\perp}\oplus W\right]  \subset H_{1}^{\perp}+T^{\ast
}W,$ so $H_{1}^{\perp}$ is $T^{\ast}$-almost invariant with defect less than
or equal to $n.$

\

\begin{lemma}
\label{lem1}Let $\Psi\in L_{B(E,E_{1})}^{\infty}$ and $\Phi\in L_{B(E_{1}%
,F)}^{\infty}.$ The following statements hold.

\begin{itemize}
\item[\textrm{(i)}] $R(T_{\Phi}H_{\Psi})^{-}$ is $S_{F}^{\ast}$-almost
invariant and $\varsigma(R(T_{\Phi}H_{\Psi})^{-})\leq\dim E_{1}.$

\item[\textrm{(ii)}] $R(H_{\Psi}^{\ast}T_{\Phi}^{\ast})^{-}$ is $S_{E}^{\ast}%
$-almost invariant and $\varsigma(R(H_{\Psi}^{\ast}T_{\Phi}^{\ast})^{-}%
)\leq\dim E_{1}$.

\item[\textrm{(iii)}] $N(T_{\Phi}H_{\Psi})$ is $S_{E}$-almost invariant and
$\varsigma(S_{E},$ $N(T_{\Phi}H_{\Psi}))=\varsigma(S_{E}^{\ast},R(H_{\Psi
}^{\ast}T_{\Phi}^{\ast})^{-})\leq\dim E_{1}.$

\item[\textrm{(iv)}] $N(H_{\Psi}^{\ast}T_{\Phi}^{\ast})$ is $S_{F}$-almost
invariant and $\varsigma(S_{F},$ $N(T_{\Phi}H_{\Psi}))=\varsigma(S_{F}^{\ast
},R(T_{\Phi}H_{\Psi})^{-})\leq\dim E_{1}.$
\end{itemize}
\end{lemma}

\begin{proof}
It follows from Lemma \ref{basic} that $S_{F}^{\ast}T_{\Phi}H_{\Psi}=T_{\Phi
}H_{\Psi}S_{E}+G$, where $G$ is a finite rank operator with rank less than or
equal to $\dim E_{1}.$ Thus $S_{F}^{\ast}R(T_{\Phi}H_{\Psi})\subset R(T_{\Phi
}H_{\Psi})+R(G).$ Since $G$ is of finite rank, by taking closure in $H_{F}%
^{2}$, one see that
\[
S_{F}^{\ast}R(T_{\Phi}H_{\Psi})^{-}\subset R(T_{\Phi}H_{\Psi})^{-}+R(G),
\]
and $R(T_{\Phi}H_{\Psi})^{-}$ is $S_{F}^{\ast}$-almost invariant and
$\varsigma(R(T_{\Phi}H_{\Psi})^{-})\leq rank(G).$ This proves (i).

By taking adjoint, we have
\[
S_{E}^{\ast}H_{\Psi}^{\ast}T_{\Phi}^{\ast}=H_{\Psi}^{\ast}T_{\Phi}^{\ast}%
S_{F}-G^{\ast}.
\]
A similar argument shows that $R(H_{\Psi}^{\ast}T_{\Phi}^{\ast})^{-}$ is
$S_{E}^{\ast}$-almost invariant and $\varsigma(R(H_{\Psi}^{\ast}T_{\Phi}%
^{\ast})^{-})\leq rank(G^{\ast}).$ This proves (ii)

Since $R(H_{\Psi}^{\ast}T_{\Phi}^{\ast})^{-}=N(T_{\Phi}H_{\Psi})^{\perp},$
by the argument just above this lemma,
 $N(T_{\Phi}H_{\Psi})$ is $S_{E}$-almost invariant and
$$
\varsigma(S_{E},N(T_{\Phi}H_{\Psi}))=\varsigma(S_{E}^{\ast},R(H_{\Psi}^{\ast
}T_{\Phi}^{\ast})^{-})\leq\dim E_{1}.
$$
This proves (iii). The proof of (iv) is similar.
\end{proof}

\begin{remark}
The above lemma also holds when $E$ and $F$ are infinite dimensional.\
\end{remark}

It is natural to ask if $R(T_{\Phi}H_{\Psi})^{-}$ is $S_{E}$-almost invariant,
the answer is yes under a more restricted assumption that $E_{1},$ $E$ and $F$
are all finite dimensional. But the defect of $R(T_{\Phi}H_{\Psi})^{-}$ is
more complicated.

\begin{lemma}
\label{lem2}
Let $\Psi\in L_{B(E,E_{1})}^{\infty}$ and $\Phi\in L_{B(E_{1}%
,F)}^{\infty}.$ Then $R(H_{\Phi}T_{\Psi})^{-}$ is $S_{F}$-almost invariant
and
\[
\varsigma(S_{F},\text{ }R((H_{\Phi}T_{\Psi})^{-}))\leq\dim E_{1}+\dim E+\dim
F.
\]
Moreover, $N(H_{\Phi}T_{\Psi})$ is $S_{E}^{\ast}$-almost invariant with
$\varsigma(S_{E}^{\ast},N(H_{\Phi}T_{\Psi}))=\varsigma(S_{E},R(T_{\Psi}^{\ast
}H_{\Phi}^{\ast})^{-}).$ In the case $\Psi^{\ast}\in H_{B(E_{1},E)}^{\infty},$
we have $\varsigma(S_{F},$ $R((H_{\Phi}T_{\Psi})^{-}) \leq
\dim E_{1}+\dim F.$
\end{lemma}

\begin{proof}
It follows from Lemma \ref{basic} that $S_{F}H_{\Phi}T_{\Psi}=H_{\Phi}T_{\Psi
}S_{E}^{\ast}+G$, where $G$ is a finite rank operator with rank less than or
equal to $\dim E_{1}+\dim E+\dim F.$ The rest of the proof is similar to the
proof of Lemma \ref{lem1}. In the case $\Psi^{\ast}\in H_{B(E_{1},E)}^{\infty
},$ we note that $H_{\Phi}S_{E_{1}}^{\ast}T_{\Psi}P_{E}=H_{\Phi}T_{\Psi}%
S_{E}^{\ast}P_{E}=0,$ so $G$ is a finite rank operator with rank less than or
equal to $\dim E_{1}+\dim F.$
\end{proof}

\

To search for better representations of $S_{E}$-almost invariant subspaces as
ranges of operators, we make another observation.

\begin{lemma}
\label{lem3}Let $\Psi\in L_{B(E,E_{1})}^{\infty}$ and $\Phi\in L_{B(E_{1}%
,F)}^{\infty}.$ Then $R(T_{\Phi}T_{\Psi})^{-}$ is $S_{F}$-almost invariant and
$\varsigma(S_{F},$ $R(T_{\Phi}T_{\Psi})^{-})\leq\dim E_{1}+\dim F.$ In
particular, if $\Phi\in H_{B(E_{1},F)}^{\infty},$ then
$\varsigma(S_{F},$ $R(T_{\Phi}T_{\Psi})^{-})\leq\dim E_{1}.$
\end{lemma}

\begin{proof}
By taking adjoint of (\ref{ts}), we have
\begin{align*}
S_{E_{1}}T_{\Psi}  &  =T_{\Psi}S_{E}-P_{E_{1}}T_{\Psi}S_{E},\\
S_{F}T_{\Phi}  &  =T_{\Phi}S_{E_{1}}-P_{F}T_{\Phi}S_{E_{1}}.
\end{align*}
Hence%
\begin{align*}
S_{F}T_{\Phi}T_{\Psi}  &  =\left(  T_{\Phi}S_{E_{1}}-P_{F}T_{\Phi}S_{E_{1}%
}\right)  T_{\Psi}\\
&  =T_{\Phi}\left(  T_{\Psi}S_{E}-P_{E_{1}}T_{\Psi}S_{E}\right)  -P_{F}%
T_{\Phi}S_{E_{1}}T_{\Psi}\\
&  =T_{\Phi}T_{\Psi}S_{E}-T_{\Phi}P_{E_{1}}T_{\Psi}S_{E}-P_{F}T_{\Phi}%
S_{E_{1}}T_{\Psi}.
\end{align*}
Since $rank(T_{\Phi}P_{E_{1}}T_{\Psi}S_{E}+P_{F}T_{\Phi}S_{E_{1}}T_{\Psi}%
)\leq\dim E_{1}+\dim F,$ by a proof similar to the proof
of Lemma \ref{lem1}, $R(T_{\Phi}T_{\Psi})^{-}$ is $S_{F}$-almost
invariant. In the case $\Phi\in H_{B(E_{1},F)}^{\infty},$ $P_{F}T_{\Phi
}S_{E_{1}}T_{\Psi}=P_{F}S_{F}T_{\Phi}T_{\Psi}=0.$
\end{proof}

\

It turns out the closure of the range of a finite sum of finite products of
Toeplitz and Hankel operators is $S_{E}$-almost invariant and $S_{E}^{\ast}%
$-almost invariant under an appropriate condition.

\begin{theorem}
\label{maina}Let $T=%
{\textstyle\sum\limits_{i=1}^{k}}
A_{i}\in B(H_{E}^{2},H_{F}^{2})$ with $A_{i}=%
{\textstyle\prod\limits_{j=1}^{m_{i}}}
C_{j,i},$ where each $C_{j,i}$ is either a (bounded) Hankel operator or a
(bounded) Toeplitz operator$.$
Then $R(T)^{-}$ is $S_{F}$-almost invariant and
$S_{F}^{\ast}$-almost invariant, and $N(T)$ is $S_{E}$-almost invariant and
$S_{E}^{\ast}$-almost invariant.
\end{theorem}

\begin{proof}
By using (\ref{ts}) and (\ref{hs}) repeatedly, if for each $i,$ $A_{i}$
contains an even number of Hankel operators, then%
\[
S_{F}^{\ast}T=%
{\textstyle\sum\limits_{i=1}^{k}}
S_{F}^{\ast}A_{i}=%
{\textstyle\sum\limits_{i=1}^{k}}
\left(  A_{i}S_{E}^{\ast}+G_{i}\right)  =TS_{E}^{\ast}+G,
\]
where each $G_{i}$ is of finite rank and $G=%
{\textstyle\sum_{i=1}^{k}}
G_{i}$ is of finite rank. Hence $R(T)^{-}$ is $S_{F}^{\ast}$-almost invariant.
Similarly, $S_{F}T=TS_{E}+G$, where $G$ is of finite rank. Hence $R(T)^{-}$ is
$S_{F}$-almost invariant. Note that $T^{\ast}=%
{\textstyle\sum_{i=1}^{k}}
A_{i}^{\ast},$ where for each $i,$ $A_{i}^{\ast}$ also contains an even number
of Hankel operators. Therefore, $R(T^{\ast})^{-}$ is $S_{E}$-almost invariant
and $S_{E}^{\ast}$-almost invariant. By Lemma \ref{duality}, $N(T)=R(T^{\ast
})^{\perp}$ is $S_{E}$-almost invariant and $S_{E}^{\ast}$-almost invariant.

If for each $i,$ $A_{i}$ contains an odd number of Hankel operators, then
$S_{F}^{\ast}T=TS_{E}+G_{1}$ and $S_{F}T=TS_{E}^{\ast}+G_{2}$, where $G_{1}$
and $G_{2}$ are of finite rank.
The same argument as the above gives the result.
\end{proof}

\

In view of the above theorem, the following question is interesting.

\begin{question}
Let $\Psi\in L_{B(E,F)}^{\infty}$ and $\Phi\in L_{B(E,F)}^{\infty}.$ When is
$R(T_{\Psi}+H_{\Phi})^{-}$ $S_{F}$-almost invariant? \ When is $R(T_{\Psi
}+H_{\Phi})^{-}$ $S_{F}^{\ast}$-almost invariant?
\end{question}

The $C^{\ast}$-algebra generated by all Toeplitz operators is called the
Toeplitz algebra and the $C^{\ast}$-algebra generated by all
Toeplitz and Hankel operators is called the Toeplitz+Hankel algebra. The
operator $T$ in Theorem \ref{maina}
belongs to this algebra.

On $H^{2},$ since $S$ is irreducible, a nontrivial $S$-invariant subspace is
not $S^{\ast}$-invariant. It seems surprising that in the examples above
$S$-almost invariant subspace is also $S^{\ast}$-almost invariant, it is
interesting to ask if this is the case for other operators.


We include the following two lemmas in this section for future use. Recall
$T\in${$B(H).$}

\begin{lemma}
If $H_{1}$ is an almost invariant subspace of $T,$ then the minimal orthogonal
defect space of $H_{1}$ is unique.
\end{lemma}

\begin{proof}
Let $W_{1}$ and $W_{2}$ be two $n$-dimensional minimal orthogonal defect
spaces of $H_{1}.$ Then $TH_{1}\subset H_{1}\oplus W_{1}$ and
$TH_{1}\subset H_{1}\oplus W_{2}.$ By the minimality of $W_{1}$, $TH_{1}%
+H_{1}\supset W_{1}.$ Hence $H_{1}\oplus W_{2}\supset TH_{1}+H_{1}\supset
W_{1}.$ This implies $W_{2}\supset W_{1}.$ Similarly, $W_{1}\supset W_{2}.$ So
$W_{1}=W_{2}.$
\end{proof}

\

We here record the argument just
above Lemma \ref{lem1} as a lemma for future use.

\begin{lemma}
\label{duality}The following statements hold.

\begin{itemize}
\item[\textrm{(i)}] $H_{1}$ is $T$-almost invariant with defect $n$ if and
only if $H_{1}^{\perp}$ is $T^{\ast}$-almost invariant with defect $n.$

\item[\textrm{(ii)}] If $TH_{1}\subset H_{1}+W$ and $\varsigma(T,H_{1})=\dim
W$, then $T\left[  H_{1}+W\right]  \subset\left[  H_{1}+W\right]  +TW$ and
$\varsigma(T,H_{1}+W)=\dim(TW)-\dim(TW\cap(W+H_{1})).$
\end{itemize}
\end{lemma}

\begin{proof}
We already argued that if $H_{1}$ is $T$-almost invariant with defect
$n,$ then $H_{1}^{\perp}$ is $T^{\ast}$-almost invariant with
defect less than or equal to $n.$ If $\varsigma(T^{\ast},H_{1}^{\perp})=k<n,$
then by applying what we just proved to $H_{1}^{\perp},$ we will get
$H_{1}=\left(  H_{1}^{\perp}\right)  ^{\perp}$ is $T$-almost invariant with
defect less than or equal to $k,$ which is a contradiction. This proves (i).

Now we prove (ii). Let $W$ be a minimal defect space of $H_{1}$.
Observe $T\left[  H_{1}+W\right]  =TH_{1}+TW\subset\left[  H_{1}+W\right]
+TW=\left[  H_{1}+W\right]  +\left[  TW\ominus(TW\cap(W+H_{1})\right]  .$ Thus
$\varsigma(T,H_{1}+W)\leq\dim(TW)-\dim(TW\cap(W+H_{1})).$ On the other hand,
assume $T\left[  H_{1}+W\right]  \subset\left[  H_{1}+W\right]  \oplus G$ for
some finite dimensional subspace $G.$ Since $W$ is a minimal defect space of
$H_{1},$ we know $TH_{1}+H_{1}\supset W.$ Thus $H_{1}+T\left[  H_{1}+W\right]
\supset W+TW$ and
\[
\left[  H_{1}+W\right]  \oplus G\supset H_{1}+T\left[  H_{1}+W\right]  \supset
H_{1}+W+TW.
\]
Hence $\dim G\geq\dim(TW)-\dim(TW\cap(W+H_{1})).$ That is, $\varsigma
(T,H_{1}+W)\geq\dim(TW)-\dim(TW\cap(W+H_{1})).$
\end{proof}

\

It is interesting to note that by iteration,
\[
\varsigma(T,H_{1})\geq\varsigma(T,H_{1}+W)\geq\varsigma(T,H_{1}+W+TW)\geq
\cdots\geq\varsigma(T,H_{1}+W+TW+\cdots+T^{k}W).
\]
So it is possible for $H_{1}+W+TW+\cdots+T^{k}W$ to become an invariant
subspace of $T.$

\section{Representations of $S$-almost invariant subspaces.}

The closely related nearly $S^{\ast}$-invariant subspaces actually predates
$S^{\ast}$-almost invariant subspaces.

\begin{definition}
A closed subspace $M$ of $H_{E}^{2}$ is said to be nearly $S_{E}^{\ast}%
$-invariant if $h\in M$ and $h(0)=0$ implies that $S_{E}^{\ast}h\in M.$
\end{definition}

Nearly $S^{\ast}$-invariant subspaces of $H^{2}$ are useful in describing
kernels of Toeplitz operators and finding inverses of Toeplitz operators
\cite{Fricain} \cite{Chevrot}. The following characterization (due to Hitt
\cite{Hitt} and Sarason \cite{Sarason}) will be useful for us, see Theorem
30.15 in \cite{Fricain}.
For $x,y\in H^\infty$, an admissible pair $(x,y)$ of an outer
function $G$ is a Pythagorean pair $(x,y)$ such that $\left\vert
x\right\vert ^{2}+\left\vert y\right\vert ^{2}=1$ on $\mathbb{T}$, $x$ is
outer and $G=x/(1-y).$

\begin{theorem}
\textrm{\cite{Hitt} \cite{Sarason}} \label{nearly}Let $M$ be a nearly
$S^{\ast}$-invariant subspace of $H^{2}.$ Then $M$ has one of the following forms:

\begin{itemize}
\item[\textrm{(i)}] There exists an inner function $\theta$ such that
$\theta(0)\neq0$ and $M=\theta H^{2}.$

\item[\textrm{(ii)}] There exists a function $g$ of unit norm in
$M$ such that $g(0)>0$ and an inner function $\theta$ such that
$\theta(0)=0$ and $\theta$ divides the function $y,$ where $(x,y)$%
is admissible pair for the outer factor $G$ of $g,$ such that $M=T_{g}%
K_{\theta},$ where $K_{\theta}=H^{2}\ominus\theta H^{2}$.
\end{itemize}
\end{theorem}

The requirement that $\theta$ divides the function $y$ is equivalent to
$T_{g}$ acts isometrically on $K_{\theta}$ so that indeed $T_{g}K_{\theta}$ is
a closed subspace of $H^{2},$ see Theorem 30.14 in \cite{Fricain}.

The case (ii) should contain the $S^{\ast}$-invariant subspace $K_{\phi}.$
Indeed, in this case if
\begin{equation}
g=(1-\left\vert \phi(0)\right\vert ^{2})^{-1/2}(1-\overline{\phi(0)}%
\phi)\ \ \text{and}\ \ \theta=\frac{\phi(0)-\phi}{1-\overline
{\phi(0)}\phi}, \label{gformula}%
\end{equation}
then $T_{g}K_{\theta}=K_{\phi}$ (cf. \cite{Sarason2}).

The vector-valued version of nearly $S^{\ast}$-invariant subspaces of
$H_{E}^{2}$, where $E$ is of finite dimension was given by
Chalendar-Chevrot-Partington \cite{ChalendarCP}. Recently the concept of
nearly $S^{\ast}$-invariant subspace with finite defect was introduced, and a
description of nearly $S^{\ast}$-invariant subspace with finite defect inside
$H^{2}$ was obtained in Chalendar-Gallardo-Partington \cite{ChalendarGP}. The
vector-valued version of nearly $S^{\ast}$-invariant subspace with finite
defect inside $H_{E}^{2}$ was described first in Chattopadhyay-Das-Pradhan
\cite{ChDas}, then independently in O'Loughlin \cite{OLoughlin}.

\begin{definition}
A closed subspace $M$ of $H_{E}^{2}$ is said to be nearly $S_{E}^{\ast}%
$-invariant with defect $p$ if there is a $p$-dimensional subspace $G$ (called a defect space of $M$ that may be taken to
be orthogonal to $M)$ such that if $h\in M$ and $h(0)=0$, then
$S_{E}^{\ast}h\in M+G.$ The smallest possible $p$ is said to be the defect of
$M,$ denoted by $\eta(S_{E}^{\ast},M)=p.$
\end{definition}

It is clear that if $M$ is $S_{E}^{\ast}$-almost invariant with $\varsigma
(S_{E}^{\ast},M)=p,$ then $M$ is nearly $S_{E}^{\ast}$-invariant with
$\eta(S_{E}^{\ast},M)\leq p.$ But if $M$ is $S_{E}^{\ast}$-almost
invariant with $\varsigma(S_{E}^{\ast},M)=p$, it is possible that $M$ is not
nearly $S_{E}^{\ast}$-invariant, see for example Proposition 2.6
in \cite{ChDas}. We will see below when a $S_{E}^{\ast}$-almost
invariant subspace is nearly $S^{\ast}$-invariant. It has been shown in
Proposition 2.2 in \cite{ChalendarGP} and Proposition 2.2
in \cite{ChDas} that a nearly $S_{E}^{\ast}$-invariant subspace
is $S_{E}^{\ast}$-almost invariant.

Next we generalize this result to a nearly $S_{E}^{\ast}$-invariant subspace
with finite defect.

\begin{proposition}
\label{nearlyalmost} If a closed subspace $M$ of $H_{E}^{2}$ is
nearly $S_{E}^{\ast}$-invariant with $\eta(S_{E}^{\ast},M)=p,$ then $M$ is
$S_{E}^{\ast}$-almost invariant with $\varsigma(S_{E}^{\ast},M)\leq p+\dim E.$
\end{proposition}

\begin{proof}
Assume $M$ is nearly $S_{E}^{\ast}$-invariant with defect $p.$ That is, there
is a $p$-dimensional subspace $W$ such that if $h\in M$ and $h(0)=0,$ then
$S_{E}^{\ast}h\in M+W.$ Set $M_{1}=\left\{  h\in M:h(0)=0\right\}  =M\cap
zH_{E}^{2}.$ Thus $S_{E}^{\ast}M_{1}\in M+W.$ Set
\[
W_{1}:=M\ominus M_{1}=M\ominus M\cap zH_{E}^{2}.
\]
By a lemma from \cite{ChalendarCP}, $\dim W_{1}\leq\dim E.$ Then%
\[
S_{E}^{\ast}M=S_{E}^{\ast}\left(  M_{1}\oplus W_{1}\right)  \subset
S_{E}^{\ast}M_{1}+S_{E}^{\ast}W_{1}\subset M+W+S_{E}^{\ast}W_{1}.
\]
Hence $\varsigma(S_{E}^{\ast},M)\leq\dim(W+S_{E}^{\ast}W_{1})\leq p+\dim E.$
\end{proof}

\

So the set of nearly $S_{E}^{\ast}$-invariant subspaces with finite defect is
the same as the set of $S_{E}^{\ast}$-almost invariant subspaces.

\begin{theorem}
\textrm{\label{nearlym}\cite{ChalendarCP} \cite{ChalendarGP} \cite{ChDas}
\cite{OLoughlin}} Let $M$ be a closed subspace of $H_{\mathbb{C}^{m}}^{2}$
that is nearly $S^{\ast}$-invariant with defect $p.$ Then:

\begin{itemize}
\item[\textrm{(i)}] In the case where there are functions in $M$ that do not
vanish at $0,$%
\begin{equation}
M=\{G:G(z)=G_{0}(z)k_{0}(z)+z%
{\textstyle\sum\limits_{i=1}^{p}}
g_{i}(z)k_{i}(z):(k_{0},\cdots,k_{p})\in K\}, \label{mrep}%
\end{equation}
where $G_{0}$ is the matrix of size $m\times r$ whose columns consist of any
orthonormal basis of $M\ominus\left(  M\cap zH_{\mathbb{C}^{m}}^{2}\right)  ,$
$\{g_{1},\cdots,g_{p}\}$ is any orthonormal basis of the defect space $W,$ and
$K\subset H_{\mathbb{C}^{r+p}}^{2}$ is a $S^{\ast}$-invariant subspace.
Furthermore, $\left\Vert G\right\Vert ^{2}=%
{\textstyle\sum_{i=0}^{p}}
\left\Vert k_{i}\right\Vert ^{2}.$

\item[\textrm{(ii)}] In the case all functions in $M$ vanish at $0,$%
\[
M=\{G:G(z)=z%
{\textstyle\sum\limits_{i=1}^{p}}
g_{i}(z)k_{i}(z):(k_{0},\cdots,k_{p})\in K\},
\]
with the same notation as in (i) except that $K\subset H_{\mathbb{C}^{p}}^{2}$
is a $S^{\ast}$-invariant subspace. Furthermore, $\left\Vert G\right\Vert
^{2}=%
{\textstyle\sum_{i=1}^{p}}
\left\Vert k_{i}\right\Vert ^{2}.$
\end{itemize}

Conversely if a closed subspace $M$ of $H_{\mathbb{C}^{m}}^{2}$ has a
representation as in \textrm{(i)} or \textrm{(ii)}, then it is nearly
$S^{\ast}$-invariant with defect $p.$
\end{theorem}

The proof of the above theorem generalizes Hitt's algorithm as in proving
Theorem \ref{nearly} and uses a lemma of \cite{Benhida} about $C_{\cdot0}$ contractions.

\begin{corollary}
\label{almostm} \textrm{\cite{ChalendarGP} \cite{ChDas} \cite{OLoughlin}} A
closed subspace $M$ of $H_{\mathbb{C}^{m}}^{2}$ is $S^{\ast}$-almost invariant
with defect $p$ if and only if it satisfies the conditions of the above
theorem together with an extra condition that the column space of $S^{\ast
}G_{0}$ is contained in $M+W$ in case \textrm{(i)}, while case \textrm{(ii)}
is unchanged.
\end{corollary}

Using Lemma \ref{lem1}, we reformulate the above theorem and corollary in
terms of ranges of Toeplitz and Hankel operators. (We have to use unbounded
Toeplitz operators, but the justification should be easy).

Recall $\Theta\in H_{B(E,E_{1})}^{\infty}$ is (left) inner if $\Theta
(z)^{\ast}\Theta(z)= I_{E}$ for almost all
$z\in\mathbb{T}$ for some $E\subset E_{1}.$ Similarly, $\Theta\in
H_{B(E,E_{1})}^{\infty}$ is right inner if $\Theta(z)\Theta(z)^{\ast
}=I_{E_{1}}$ for almost all  $z\in\mathbb{T}$ for
some $E_{1}\subset E.$ When $E=E_{1}$ and $\Theta$ is both left and right
inner, we say $\Theta$ is a two-sided inner function or a square inner
function. In short, if $\Theta$ is left inner, then we just say $\Theta$ is
inner. If $\Theta\in H_{B(E,E_{1})}^{\infty}$ is inner, let $K_{\Theta
}=H_{E_{1}}^{2}\ominus\Theta H_{E}^{2}$ denote the model space. The celebrated
Beurling-Lax-Halmos (BLH) Theorem \cite{Beurling} \cite{Lax} \cite{Halmos}
says an invariant subspace of $S_{E_{1}}$ is of the form $\Theta H_{E}^{2},$
and consequently, an invariant subspace of $S_{E_{1}}^{\ast}$ is of the form
$K_{\Theta}.$ It follows that $T_{\Theta}$ is an isometry from $H_{E}^{2}$
into $H_{E_{1}}^{2}$. When $\Theta$ is two-sided inner, by (\ref{thc}),
$H_{\Theta^{\ast}}^{\ast}H_{\Theta}^{\ast}=I-T_{\Theta}T_{\Theta
}^{\ast},$ so $H_{\Theta^{\ast}}^{\ast}$ is a partial isometry and
$R(H_{\Theta^{\ast}}^{\ast})=K_{\Theta}.$ It is known that the kernel of a
Hankel operator is $S$-invariant. However, when $\Theta$ is just left inner,
$N(H_{\Theta^{\ast}})\supset\Theta H_{E}^{2}$ and $R(H_{\Theta^{\ast}}^{\ast
})\subset K_{\Theta}$ in general. Thus representing $\Theta H_{E}^{2}$ as the
kernel of a Hankel operator is studied in \cite{GuDOParkj} when $\dim
E_{1}<\infty$, see also \cite{CHL} for a study of this problem when $\dim
E_{1}=\infty.$ See also \cite{LR15} where the connection of Hankel operators
and shift invariant subspaces on Dirichlet spaces are studied.
Connected to this, we have the following result.

\begin{lemma}
\textrm{\cite{CHL}}\label{hnull} Let $\Theta\in H_{B(E,E_{1})}^{\infty}$ be
inner. Then there exists $\Phi\in L_{B(E_{1},E_{1})}^{\infty}$ such that
$N(H_{\Phi})=\Theta H_{E}^{2}.$ Consequently, $K_{\Theta}=R(H_{\widetilde
{\Phi}})^{-}.$
\end{lemma}

$K_{\Theta}$ is unchanged if we replace $\Theta$ by $\Theta\oplus U,$
where $U$ is a constant unitary operator. To clarify this situation, we recall
a classical concept. Let $B\in H_{B(E,E_{1})}^{\infty}$ and $B$ is
a contractive operator-valued function. The function $B$ is called purely
contractive (or just pure) if $\left\Vert B(0)v\right\Vert <\left\Vert
v\right\Vert $ for all $v\in E\backslash\{0\}.$ Any contractive
$B\in H^{\infty}_{B(E, E_{1})}$ admits a decomposition such that
\begin{equation}
B(z)=B_{1}(z)\oplus U:E=E_{1}\oplus E_{2}\rightarrow F=F_{1}\oplus F_{2},
\label{pured}%
\end{equation}
where $B_{1}$ is $B(E_{1},F_{1})$-valued and pure in the sense
that $\left\Vert B_{1}(0)v\right\Vert <\left\Vert v\right\Vert $ for $v\in
E_{1}\backslash\{0\}$ and $U$ is a unitary constant from $E_{2}$ onto $F_{2}.$
The $B_{1}$ is referred to be the purely contractive part and $U$ is the
unitary part of $B,$ see \cite[Page 194]{NFBK}. It is easy to see that with
respect to the decomposition (\ref{pured}),
\[
I_{F}-T_{B}T_{B}^{\ast}=\left(  I_{F_{1}}-T_{B_{1}%
}T_{B_{1}}^{\ast}\right)  \oplus0_{F_{2}}.
\]
Therefore, to represent $K_{\Theta},$ we can assume $\Theta$ is inner and pure.

\begin{theorem}
\label{main}A closed subspace $M$ of $H_{F}^{2}$ is $S_{F}^{\ast}$-almost
invariant if and only if one of the following holds:

\begin{itemize}
\item[\textrm{(1)}] $M=R(T_{\Theta})$, where $\Theta\in H_{B(E,F)}^{\infty}$
is inner. In this case, $\varsigma(S_{F}^{\ast},$ $M)=\dim E-rank(U),$ where
$U$ is the unitary part of $\Theta$.

\item[\textrm{(2)}] $M=R(T_{\Phi}\left(  I_{E_{1}}-T_{\Theta}T_{\Theta}^{\ast
}\right)  )$, where $\Theta\in H_{B(E,E_{1})}^{\infty}$ is inner and pure,
$\Phi\in H_{B(E_{1},F)}^{2},$ $\dim E_{1}<\infty,$ and $T_{\Phi}\left(
I_{E_{1}}-T_{\Theta}T_{\Theta}^{\ast}\right)  $ is a partial isometry. In this
case, a minimal defect space is $W:=R(S_{F}^{\ast}T_{\Phi}P_{E_{1}}%
)\ominus\left[  R(S_{F}^{\ast}T_{\Phi}P_{E_{1}})\cap M\right]  $ and
$\varsigma(S_{F}^{\ast},$ $M)=\dim W.$
\end{itemize}
\end{theorem}

\begin{proof}
The ``if" direction follows from Theorem \ref{maina}. For the "only if"
direction, set $F=\mathbb{C}^{m}.$ We will deal with case (i) in Theorem
\ref{nearlym} since the proof of case (ii) is similar. Then in case (i)
$E_{1}=\mathbb{C}^{r+p}$, $\Phi=\left[
\begin{array}
[c]{cccc}%
G_{0} & zg_{1} & \cdots & zg_{p}%
\end{array}
\right]  $ is an inner matrix function  of size $m\times(r+p).$
Since $K\subset H_{\mathbb{C}^{r+p}}^{2}$ is $S^{\ast}$-invariant, either
$K=H_{\mathbb{C}^{r+p}}^{2}$ and hence, $M=R(T_{\Phi})$ and (1)
holds, or by BLH\ Theorem, $K=K_{\Theta}=H_{\mathbb{C}^{r+p}}%
^{2}\ominus\Theta H_{\mathbb{C}^{n}}^{2},$ where $\Theta\in H_{B(E,E_{1}%
,)}^{\infty}$ for $E=\mathbb{C}^{n}$ with $n\leq r+p.$ {\color{red}Note that} $I-T_{\Theta
}T_{\Theta}^{\ast}$ is the projection from $H_{\mathbb{C}^{r+p}}^{2}$ onto
$K_{\Theta}.$ By (\ref{mrep}), $M=T_{\Phi}K_{\Theta}=R(T_{\Phi}\left(
I_{E_{1}}-T_{\Theta}T_{\Theta}^{\ast}\right)  ).$ Now the equation $\left\Vert
G\right\Vert ^{2}=%
{\textstyle\sum_{i=0}^{p}}
\left\Vert k_{i}\right\Vert ^{2}$ is equivalent to that $T_{\Phi}$ acts on
$K_{\Theta}$ as an isometry. Namely, for all $h\in H_{E_{1}}^{2},$%
\[
\left\langle T_{\Phi}\left(  I-T_{\Theta}T_{\Theta}^{\ast}\right)  h,T_{\Phi
}\left(  I-T_{\Theta}T_{\Theta}^{\ast}\right)  h\right\rangle =\left\langle
\left(  I-T_{\Theta}T_{\Theta}^{\ast}\right)  h,h\right\rangle .
\]
Equivalently,%
\begin{equation}
\left[  T_{\Phi}\left(  I-T_{\Theta}T_{\Theta}^{\ast}\right)  \right]  ^{\ast
}\left[  T_{\Phi}\left(  I-T_{\Theta}T_{\Theta}^{\ast}\right)  \right]
=\left(  I-T_{\Theta}T_{\Theta}^{\ast}\right)  . \label{product}%
\end{equation}
Since $\left(  I-T_{\Theta}T_{\Theta}^{\ast}\right)  $ is a projection, we
know $\left[  T_{\Phi}\left(  I-T_{\Theta}T_{\Theta}^{\ast}\right)  \right]  $
is a partial isometry. The proof of the ``only if" direction is complete.

If $M=R(T_{\Theta}),$ then by (\ref{ts}), we have%
\[
S_{F}^{\ast}T_{\Theta}=T_{\Theta}S_{E}^{\ast}+S_{F}^{\ast}T_{\Theta}P_{E}.
\]
Thus the defect space $W$ of $M=R(T_{\Theta})$ is a subspace $R(S_{F}^{\ast
}T_{\Theta}P_{E}),$ where%
\[
R(S_{F}^{\ast}T_{\Theta}P_{E})=Span\left\{  \overline{z}\left[  \Theta
(z)-\Theta(0)\right]  e:e\in E\right\}  .
\]
Note that $\overline{z}\left[  \Theta(z)-\Theta(0)\right]  e\in
R(T_{\Theta})^{\perp}$, and hence the minimal orthogonal defect
space $W$ is $R(S_{F}^{\ast}T_{\Theta}P_{E}).$ It is easy to see that $\dim
R(S_{F}^{\ast}T_{\Theta}P_{E})=\dim E-rank(U),$ where $U$ is the unitary part
of $\Theta(z).$

If $M=R(T_{\Phi}\left(  I_{E_{1}}-T_{\Theta}T_{\Theta}^{\ast}\right)  ),$ then
by using (\ref{ts}) twice, we have
\begin{align}
&  S_{F}^{\ast}T_{\Phi}\left(  I-T_{\Theta}T_{\Theta}^{\ast}\right)
\nonumber\label{add1}\\
&  =\left(  T_{\Phi}S_{E_{1}}^{\ast}+S_{F}^{\ast}T_{\Phi}P_{E_{1}}\right)
\left(  I-T_{\Theta}T_{\Theta}^{\ast}\right)  \nonumber\\
&  =T_{\Phi}S_{E_{1}}^{\ast}\left(  I-T_{\Theta}T_{\Theta}^{\ast}\right)
+S_{F}^{\ast}T_{\Phi}P_{E_{1}}\left(  I-T_{\Theta}T_{\Theta}^{\ast}\right)
\nonumber\\
&  =T_{\Phi}\left(  I-T_{\Theta}T_{\Theta}^{\ast}\right)  S_{E_{1}}^{\ast}+G,
\end{align}
where%
\begin{equation}
G=-T_{\Phi}S_{E_{1}}^{\ast}T_{\Theta}P_{E}T_{\Theta}^{\ast}+S_{F}^{\ast
}T_{\Phi}P_{E_{1}}\left(  I-T_{\Theta}T_{\Theta}^{\ast}\right)  .\label{defGG}%
\end{equation}
Note $R(T_{\Phi}S_{E_{1}}^{\ast}T_{\Theta}P_{E}T_{\Theta}^{\ast
})\subset R(T_{\Phi}S_{E_{1}}^{\ast}T_{\Theta}P_{E})$ and
\[
R(T_{\Phi}S_{E_{1}}^{\ast}T_{\Theta}P_{E})=\left\{  \Phi(z)\overline{z}\left[
\Theta(z)-\Theta(0)\right]  e:e\in E\right\}  .
\]
Since $\Phi(z)\overline{z}\left[  \Theta(z)-\Theta(0)\right]  e\in M,$ the
defect space $W$ is a subspace of $R(S_{F}^{\ast}T_{\Phi}P_{E_{1}}\left(
I-T_{\Theta}T_{\Theta}^{\ast}\right)  )$. We claim that if
$\Theta$ is pure, then
\begin{equation}
R(P_{E_{1}}\left(  I-T_{\Theta}T_{\Theta}^{\ast}\right)  )=R(P_{E_{1}%
}).\label{fulle1}%
\end{equation}
Thus a minimal defect space $W$ is $R(S_{F}^{\ast}T_{\Phi}P_{E_{1}}%
)\ominus\left[  R(S_{F}^{\ast}T_{\Phi}P_{E_{1}})\cap M\right]  .$ To prove the
claim, assume there exists $e_{0}\in E_{1}$ such that $e_{0}\perp R(P_{E_{1}%
}\left(  I-T_{\Theta}T_{\Theta}^{\ast}\right)  ).$ Then for any $e_{1}\in
E_{1}$
\begin{align*}
0 &  =\left\langle e_{0},P_{E_{1}}\left(  I-T_{\Theta}T_{\Theta}^{\ast
}\right)  e_{1}\right\rangle _{E_{1}}\\
&  =\left\langle e_{0},e_{1}-P_{E_{1}}\left(  T_{\Theta}\Theta(0)^{\ast}%
e_{1}\right)  \right\rangle _{E_{1}}\\
&  =\left\langle e_{0},\left(  I_{E_{1}}-\Theta(0)\Theta(0)^{\ast})\right)
e_{1}\right\rangle _{E_{1}}.
\end{align*}
Since $\Theta$ is pure, $\Theta(0)$ is a strict contraction. Hence $e_{0}=0$
and the claim is proved.
\end{proof}

\

For a scalar Toeplitz operator $T_{\varphi}$, where $\varphi\in L^{\infty},$
it is known that $T_{\varphi}$ is a partial isometry if and only if $\varphi$
is inner \cite{BrownDouglas}. But it is difficult to characterize when
$T_{\Phi}\left(  I_{E_{1}}-T_{\Theta}T_{\Theta}^{\ast}\right)  $ is a partial
isometry, see Theorem \ref{nearly} above for the case when
both $\Phi\equiv g$ and
$\Theta\equiv \theta$ are scalar-valued functions.

\begin{corollary}
A closed subspace $M$ of $H_{F}^{2}$ is $S_{F}^{\ast}$-almost invariant if and
only if either $M=R(T_{\Theta})$, where $\Theta\in H_{B(E,F)}^{\infty}$ is
inner or $M=R(T_{\Phi}H_{\Psi})^{-}$, where $\Phi\in H_{B(E,F)}^{\infty}$ and
$\Psi\in L_{B(F,E)}^{\infty}.$
\end{corollary}

\begin{proof}
By the above theorem, either $M=R(T_{\Theta})$ or $M=T_{\Phi}K_{\Theta}.$ By
Lemma \ref{hnull}, $K_{\Theta}=R(H_{\Psi})^{-}$ for some $\Psi\in
L_{B(F,E)}^{\infty}.$ Thus $M=T_{\Phi}K_{\Theta}=R(T_{\Phi
}H_{\Psi})^{-}.$
\end{proof}

\

The above theorem seems to suggest that ``an extra condition that the
column space of $S_{F}^{\ast}G_{0}$ is contained
in $M+G$ in case (i)" in Corollary
\ref{almostm} is not needed which is not the case, what the above theorem
says is
that without this condition, $M$ is still $S_{F}^{\ast}$-almost
invariant, but the defect of $M$ is more than $p.$ Corollary \ref{almostm}
shows how to get a $S_{F}^{\ast}$-almost invariant subspace with defect $p$
from a nearly $S_{F}^{\ast}$-invariant subspace with defect $p$ by this extra
condition. Our theorem combines these two concepts using Proposition
\ref{nearlyalmost}. Theorem \ref{main} captures the essential part of Theorem
\ref{nearlym} and Corollary \ref{almostm} and leaves out some details. We
think such a reformulation and simplification is useful. The approach to view
$M$ as the range of an operator involving Toeplitz and Hankel operators is
fruitful. Of course we can also add back more details (a more precise
information of $\Phi$) if we need to. We can pick out nearly $S_{F}^{\ast}%
$-invariant subspaces from $S_{F}^{\ast}$-almost invariant subspaces in the
following way.

\begin{corollary}
Let $M:=R(T_{\Phi}\left(  I_{E_{1}}-T_{\Theta}T_{\Theta}^{\ast}\right)  )$ be
as in Theorem \ref{main} \textrm{(2)}. Then $M$ is nearly $S_{F}^{\ast}%
$-invariant if and only if $rank\left[  \Phi(0)\right]  =\dim E_{1}.$ In
particular, $\dim E_{1}\leq\dim F.$ If $M=R(T_{\Theta})$, where $\Theta\in
H_{B(E,F)}^{\infty}$ is inner, then $M$ is nearly $S_{F}^{\ast}%
$-invariant if and only if $rank\left[  \Theta(0)\right]  =\dim F.$ In
particular, $\Theta$ is two-sided inner.
\end{corollary}

\begin{proof}
Assume $rank\left[  \Phi(0)\right]  =\dim E_{1}.$ Let $h\in M$ be such that
$h(0)=0.$ Write $h=\Phi k,$ where $k\in K_{\Theta}.$
Then $0=h(0)=\Phi(0)k(0)$ implies $k(0)=0$ since $\Phi(0)$ is full column
rank. Hence $S_{F}^{\ast}h=\Phi S_{E_{1}}^{\ast}k \in M$ and $M$
is nearly $S_{F}^{\ast}$-invariant.

On the other hand, assume $M$ is nearly $S_{F}^{\ast}$-invariant. Let $h\in M$
be such that $h(0)=0.$ Write $h=\Phi k,$ where $k\in
K_{\Theta}.$ Then%
\begin{align*}
S_{F}^{\ast}h(z)  &  =\Phi(z)S_{E_{1}}^{\ast}k(z)+\overline{z}\left[
\Phi(z)-\Phi(0)\right]  k(0)\\
&  =\Phi(z)S_{E_{1}}^{\ast}k(z)+\overline{z}\Phi(z)k(0)\in M
\end{align*}
implies that $\overline{z}\Phi(z)k(0)=\Phi(z)k_{1}(z)$ for some
$k_{1}\in K_{\Theta}$. Thus $\Phi(z)k(0)=\Phi(z)zk_{1}(z)$ and
$k(0)=zk_{1}(z).$ This can only happen if $k(0)=0.$ In conclusion,
$\Phi(0)k(0)=0$ implies $k(0)=0.$ If we show the set of all the possible
$k(0)$ is $E_{1},$ then the above implication proves that $rank\left[
\Phi(0)\right]  =\dim E_{1}.$ Indeed, the set of all possible $k(0)=P_{E_{1}%
}k_{1}(z),$ where $k_{1}\in K_{\Theta},$ is $R(P_{E_{1}}\left(
I-T_{\Theta}T_{\Theta}^{\ast}\right)  ),$ which is equal to $R(P_{E_{1}})$ by
(\ref{fulle1}) since $\Theta$ is pure.

In the case $M=R(T_{\Theta}),$ similarly, we can prove that $M$ is nearly
$S_{F}^{\ast}$-invariant if and only if $rank\left[  \Theta(0)\right]  =\dim
F.$
\end{proof}

\

The above corollary contains Lemma 2.4 in \cite{ChDas}, where $\Phi$ is
assumed to be a diagonal inner function. By Theorem \ref{main} and Lemma
\ref{lem2}, a $S_{F}^{\ast}$-almost invariant subspace is $S_{F}%
$-almost invariant, special examples in the scalar-valued case $(\dim F=1)$
have been observed in Proposition 2.2 in \cite{ChalendarGP} and
in the vector-valued case in Proposition 2.2 in \cite{ChDas}. We also
have the converse.

\begin{corollary}
\label{equivalent}Let $M$ be a closed subspace of $H_{F}^{2}.$ Then $M$ is
$S_{F}^{\ast}$-almost invariant if and only if $M$ is $S_{F}$-almost invariant.
\end{corollary}

\begin{proof}
If $M$ is $S_{F}^{\ast}$-almost invariant, then by the above theorem,
$M=R(T_{\Phi}\left(  I_{E_{1}}-T_{\Theta}T_{\Theta}^{\ast}\right)  ).$ It
follows from Theorem \ref{maina} that $M$ is $S_{F}$-almost invariant.

If $M$ is $S_{F}$-almost invariant, then by Lemma \ref{duality}, $M^{\perp}$
is $S_{F}^{\ast}$-almost invariant. By what we just proved, $M^{\perp}$ is
$S_{F}$-almost invariant. By Lemma \ref{duality} again, $M=\left(  M^{\perp
}\right)  ^{\perp}$ is $S_{F}^{\ast}$-almost invariant.
\end{proof}

\

Quite amazingly, we obtain the same characterization of $S_{F}$-almost
invariant subspaces. We also find the defect space.

\begin{theorem}
\label{main2} Under the same notation as in Theorem \ref{main}, a
closed subspace $M$ of $H_{F}^{2}$ is $S_{F}$-almost invariant if and only if
one of the following holds:

\begin{itemize}
\item[\textrm{(1)}] $M=R(T_{\Theta})$, where $\Theta\in H_{B(E,F)}^{\infty}$
is inner. In this case, $\varsigma(S_{F},$ $M)=0.$

\item[\textrm{(2)}] $M=R(T_{\Phi}\left(  I_{E_{1}}-T_{\Theta}T_{\Theta}^{\ast
}\right)  )$, where $\Theta\in H_{B(E,E_{1})}^{\infty}$ is inner and pure,
$\Phi\in H_{B(E_{1},F)}^{2},$ $\dim E_{1}<\infty,$ and $T_{\Phi}\left(
I_{E_{1}}-T_{\Theta}T_{\Theta}^{\ast}\right)  $ is a partial isometry. In this
case, a minimal defect space $W$ is $R(T_{\Phi}T_{\Theta}P_{E})\ominus\left[
R(T_{\Phi}T_{\Theta}P_{E})\cap M\right]  $ and $\varsigma(S_{F},$ $M)=\dim W.$
\end{itemize}
\end{theorem}

\begin{proof}
By using (\ref{ts}), we have
\begin{align}
&  S_{F}T_{\Phi}\left(  I-T_{\Theta}T_{\Theta}^{\ast}\right)  \nonumber\\
&  =T_{\Phi}S_{E_{1}}-T_{\Phi}T_{\Theta}S_{E}T_{\Theta}^{\ast}\nonumber\\
&  =T_{\Phi}S_{E_{1}}-T_{\Phi}T_{\Theta}\left(  T_{\Theta}^{\ast}S_{E_{1}%
}-P_{E}T_{\Theta}^{\ast}S_{E_{1}}\right)  \nonumber\\
&  =T_{\Phi}\left(  I-T_{\Theta}T_{\Theta}^{\ast}\right)  S_{E_{1}}+T_{\Phi
}T_{\Theta}P_{E}T_{\Theta}^{\ast}S_{E_{1}}.\label{defGGG}%
\end{align}
Therefore, the defect space $W$ is a subspace of $R(T_{\Phi}T_{\Theta}%
P_{E})=\left\{  \Phi(z)\Theta(z)e:e\in E\right\}.$  We now claim $R(P_{E}%
T_{\Theta}^{\ast}S_{E_{1}})=E.$ To prove the claim, assume there
exists $e_{0}\in E$ such that $e_{0}\perp R(P_{E}T_{\Theta}^{\ast}S_{E_{1}}).$
Then for any $h\in H_{E_{1}}^{2},$
\begin{align*}
0 &  =\left\langle e_{0},P_{E}T_{\Theta}^{\ast}S_{E_{1}}h\right\rangle
_{E}=\left\langle e_{0},P_{E}T_{\Theta}^{\ast}S_{E_{1}}h\right\rangle
_{H_{E}^{2}}\\
&  =\left\langle e_{0},T_{\Theta}^{\ast}S_{E_{1}}h\right\rangle _{H_{E}^{2}%
}=\left\langle \Theta e_{0},zh\right\rangle _{H_{E}^{2}}\\
&  =\left\langle \overline{z}\left[  \Theta(z)-\Theta(0)\right]
e_{0},h\right\rangle _{H_{E}^{2}}.
\end{align*}
Thus $\overline{z}\left[  \Theta(z)-\Theta(0)\right]  e_{0}=0$ and $\left\Vert
\Theta(z)e_{0}\right\Vert _{H_{E}^{2}}=\left\Vert \Theta(0)e_{0}\right\Vert
_{H_{E}^{2}}.$ Since $T_{\Theta}$ is an isometry, $\left\Vert \Theta
(z)e_{0}\right\Vert _{H_{E}^{2}}=\left\Vert e_{0}\right\Vert _{E}.$ Since
$\Theta$ is pure, $\left\Vert \Theta(0)e_{0}\right\Vert _{H_{E}^{2}%
}<\left\Vert e_{0}\right\Vert _{E}\ $if $e_{0}\neq0$. Hence $e_{0}=0$ and the
claim is proved. Thus a minimal defect space $W$ is $R(T_{\Phi}T_{\Theta}%
P_{E})\ominus\left[  R(T_{\Phi}T_{\Theta}P_{E})\cap M\right]  .$
\end{proof}


\section{Examples of $S_{F}$-almost invariant subspaces}

In spite of a complete characterization of $S_{F}^{\ast}$-almost invariant
subspaces and $S_{F}$-almost invariant subspaces by \cite{ChalendarCP}
\cite{ChalendarGP} \cite{ChDas} \cite{OLoughlin} and by Theorem \ref{main} and
Theorem \ref{main2}, a good understanding of those subspaces requires to
clarify the relation between $\Phi$ and $\Theta$ which seems difficult, as
this can be already seen in Theorem \ref{nearly} which is the simplest case
when both $\Phi\equiv g$ and $\Theta\equiv \theta.$
are scalar-valued functions. In this section, we
give some examples of such $\Phi$ and $\Theta$.

For $a\in\mathbb{D}$, let
\begin{equation}
\varphi_{a}(z)=\frac{z-a}{1-\overline{a}z} \label{defm}%
\end{equation}
be the automorphism of $\mathbb{D}$. Then $b(z)=\lambda%
{\textstyle\prod_{i=1}^{n}}
\varphi_{a_{i}}(z)$ is a finite Blaschke product, where $a_{i}\in\mathbb{D}$
and $\left\vert \lambda\right\vert =1.$ Let $F$ be a subspace of
$E$ and $P_{F}$ be the projection from $E$ onto $F.$ The inner function
\[
Q(z,a,F,E)=\varphi_{a}(z)\left(  I_{E}-P_{F\text{ }}\right)  +P_{F}%
\]
is called a Blaschke-Potapov factor. To avoid triviality, we assume $F\neq
E\ $whenever we write down a Blaschke-Potapov factor. Let $B(z):=U%
{\textstyle\prod_{i=1}^{n}}
Q(z,a_{i},F_{i},E)$, where $U\in B(E)$ is unitary, $a_{i}\in\mathbb{D}$ and
$F_{i}$ is a subspace of $E$ for $1\leq i\leq n.$ Such a $B\in
H_{B(E)}^{\infty}$ is called a finite Blaschke-Potapov product. It is known
that $K_{\Theta}$ is of finite dimension if and only if $\Theta$ is a finite
Blaschke-Potapov product \cite{Pe}.

\begin{proposition}\label{proposition4.1}
Let $\Phi\in H_{B(E_{1},F)}^{2}$ be inner and $\Theta\in H_{B(E,E_{1}%
)}^{\infty}$ be inner and pure. Let
\[
M:=R(T_{\Phi}\left(  I_{E_{1}}-T_{\Theta}T_{\Theta}^{\ast}\right)  ).
\]
The following statements hold.

\begin{itemize}
\item[\textrm{(i)}] $M$ is $S_{F}$-almost invariant and $S_{F}^{\ast}$-almost invariant.

\item[\textrm{(ii)}] $\varsigma(S_{F},$ $M)=\varsigma(S_{F}^{\ast},M^{\perp
})=\dim E$ and $\varsigma(S_{F},$ $M^{\perp})=\varsigma(S_{F}^{\ast},$
$M)=\dim E_{1}-rank(U),$ where $U$ is the unitary part of $\Phi.$

\item[\textrm{(iii)}] $M^{\perp}=R(T_{\Phi\Theta})\oplus K_{\Phi}%
=R(T_{\Phi_{1}}\left(  I-T_{\Theta_{1}}T_{\Theta_{1}}^{\ast}\right)  )$, where
$\Phi_{1}$ and $\Theta_{1}$ are defined by
\begin{equation}
\Phi_{1}=\left[
\begin{array}
[c]{cc}%
\Phi\Theta & I_{F}%
\end{array}
\right]  \in H_{B(E\oplus F,F)}^{\infty}\quad\text{and}\quad\Theta_{1}=\left[
\begin{array}
[c]{c}%
0\\
\Phi
\end{array}
\right]  \in H_{B(E,E\oplus F)}^{\infty}\text{.} \label{phitheta}%
\end{equation}

\item[\textrm{(iv)}] $M$ is a half-space if and only if $\Theta$ is not a
finite Blaschke-Potapov product.
\end{itemize}
\end{proposition}

\begin{proof}
By assumption, we have $E\subset E_{1}\subset F.$
Since $\Phi$ is inner, by Theorem \ref{main2}, $\varsigma(S_{F},$
$M)=\dim R(T_{\Phi}T_{\Theta}P_{E}).$ But $T_{\Phi}T_{\Theta}P_{E}=T_{\Phi}%
{}_{\Theta}P_{E}$ and $\Phi\Theta$ is inner, so $\dim R(T_{\Phi}T_{\Theta
}P_{E})=\dim E.$ By Lemma \ref{duality}, $\varsigma(S_{F}^{*},
\,M^{\perp})=\varsigma(S_{F},$ $M)=\dim E.$

By Theorem \ref{main}, $\varsigma(S_{F}^{\ast},$ $M)=\dim W,$ where
$W:=R(S_{F}^{\ast}T_{\Phi}P_{E_{1}})\ominus\left[  R(S_{F}^{\ast}T_{\Phi
}P_{E_{1}})\cap M\right]  $ is a minimal defect space. We claim $R(S_{F}%
^{\ast}T_{\Phi}P_{E_{1}})\cap M=\{0\}.$ Assume $e_{1}\in E_{1}$ be such that
for some $h\in H_{E_{1}}^{2}$%
\[
S_{F}^{\ast}T_{\Phi}P_{E_{1}}e_{1}=T_{\Phi}\left(  I_{E_{1}}-T_{\Theta
}T_{\Theta}^{\ast}\right)  h.
\]
Set $g:=\left(  I_{E_{1}}-T_{\Theta}T_{\Theta}^{\ast}\right)  h.$ It follows
that%
\[
\left[  \Phi(z)-\Phi(0)\right]  e_{1} =z\Phi(z)g\ \ \text{and}\ \ \Phi
(z)\left(  e_{1}-zg\right)  =\Phi(0)e_{1}.
\]
Since $\Phi$ is inner,
\[
\left\Vert \Phi(0)e_{1}\right\Vert ^{2} = \left\Vert
\Phi(z)\left(  e_{1}-zg\right)  \right\Vert ^{2} =\left\Vert \left(
e_{1}-zg\right)  \right\Vert ^{2} =\left\Vert e_{1}\right\Vert ^{2}+\left\Vert
zg\right\Vert ^{2}.
\]
Thus $g=0.$ This proves the claim $R(S_{F}^{\ast}T_{\Phi}P_{E_{1}})\cap
M=\{0\}.$ So $W=R(S_{F}^{\ast}T_{\Phi}P_{E_{1}})$ is a minimal defect space.
Write $\Phi=U\oplus \Psi$, where $U$ is the unitary part of $\Phi$
and $\Psi$ is the purely contractive part of $\Phi.$ Therefore,
$W=R(S_{F}^{\ast}T_{\Phi}P_{E_{1}})=R(S_{F}^{\ast} T_{0\oplus\Psi
}P_{E_{1}})$ and $\dim W=\dim E_{1}-rank(U).$ This proves (ii).

Next for (iii), we find\textbf{ }$M^{\perp}.$ Assume $h\perp M$.
Then $T_{\Phi}^{\ast}h\perp K_{\Theta}$ and there exists $g\in H_{E}^{2}$ such
that $T_{\Phi}^{\ast}h=\Theta g.$ Thus%
\[
\Phi^{\ast}h=\Theta g+\overline{zu(z)}\text{ for some }u\in H_{E_{1}}^{2}.
\]
Write $h=\Phi h_{1}+h_{2},$ where $h_{1}\in H_{E_{1}}^{2}$ and $h_{2}\in
K_{\Phi}.$ Plugging this decomposition of $h$ into the above equation, we
have
\[
h_{1}=\Theta g+\overline{zu(z)}+\Phi^{\ast}h_{2}.
\]
Since $\Phi^{\ast}h_{2}\in\overline{zH_{E_{1}}^{2}},$ we have
$h_{1}=\Theta g.$ Hence $h=\Phi h_{1}+h_{2}=\Phi\Theta g+h_{2}\in
R(T_{\Phi\Theta})\oplus K_{\Phi}$ and $M^{\perp}\subset R(T_{\Phi\Theta
})\oplus K_{\Phi}.$ The inclusion $R(T_{\Phi\Theta})\oplus K_{\Phi}\subset
M^{\perp}$ can be verified. So $M^{\perp}=R(T_{\Phi\Theta})\oplus K_{\Phi}.$
According to Theorem \ref{main2}, $R(T_{\Phi\Theta})\oplus K_{\Phi}%
=R(T_{\Phi_{1}}\left(  I-T_{\Theta_{1}}T_{\Theta_{1}}^{\ast}\right)  )$ for
some $\Phi_{1}$ and $\Theta_{1}.$ Indeed, if $\Phi_{1}$ and
$\Theta_{1}$ are defined by (\ref{phitheta}), then $\Theta_{1}$ is inner and
$T_{\Phi_{1}}$ acts an an isometry on $K_{\Theta_{1}}.$ Since $K_{\Theta_{1}%
}=H_{E}^{2}\oplus K_{\Phi},$ we have $M^{\perp}=R(T_{\Phi_{1}}\left(
I-T_{\Theta_{1}}T_{\Theta_{1}}^{\ast}\right)  ).$ We can also use this
representation $M^{\perp}$ to compute $\varsigma(S_{F}^{\ast},M^{\perp})$ and
$\varsigma(S_{F},$ $M^{\perp})$ directly. This proves (iii).

Since $T_{\Phi}$ is an isometry, $M$ is finite dimensional if and only only if
$\Theta$ is a finite Blaschke-Potapov product. When $\Theta$ is not a finite
Blaschke-Potapov product, then $M$ is infinite dimensional. By (iii),
$M^{\perp}$ is always infinite dimensional. Hence $M$ is a half-space if and
only if $\Theta$ is not a finite Blaschke-Potapov product. This
proves (iv).
\end{proof}

\

\begin{example}\label{example4.2} From Proposition \ref{proposition4.1}, one can see
that if $\Phi\in H^\infty_{B(E)}$ is inner and $\Theta\in
H^\infty_{B(E)}$ is inner and pure then
$$
\Phi\bigl( \hbox{\rm ran}\,(I_E-T_\Theta T_\Theta^*)\bigr)= \Phi
\bigl(\hbox{\rm ran}\, (H_{\Theta^*}^*)\bigr)= \Phi K_\Theta
$$
is $S_F$-almost invariant. In particular, $K_{\Theta}$ is
$S_{F}$-almost invariant with $\varsigma(S_{F},$ $K_{\Theta })=\dim
E.$
\end{example}

Example \ref{example4.2} contains Propositions 1.4 and
1.5 in \cite{ChalendarGP}, where $M=\varphi K_{\theta}$ is
considered with $\varphi$ and $\theta$ being two scalar inner
functions and it also contains Propositions 2.5 and 2.6 in
\cite{ChDas} where $M=\left( \Psi K_{\Theta}\right)  ^{\perp}$ is
considered with $\Psi$ being a diagonal square inner function and
$\Theta$ being an inner function.

\

In view of Proposition \ref{proposition4.1}, we make the following
conjecture.

\begin{conjecture}
Let $M:=R(T_{\Phi}\left(  I_{E_{1}}-T_{\Theta}T_{\Theta}^{\ast}\right)  )$ be
as in Theorem \ref{main}. Then $M$ is a half-space if and only if $\Theta$ is
not a finite Blaschke-Potapov product.
\end{conjecture}

It follows from Lemma \ref{lem1} (iv) that we can characterize a $S_{F}%
$-almost invariant subspace as the kernel of $H_{\Theta^{\ast}}T_{\Phi}^{\ast
}$ (when $T_{\Phi}^{\ast}$ is unbounded, we simply interpret $\ker
H_{\Theta^{\ast}}T_{\Phi}^{\ast}$ as $\left[  R(T_{\Phi}H_{\Theta^{\ast}%
}^{\ast})\right]  ^{\perp}$).

\begin{corollary}
Under the same notation as in Theorem \ref{main}, a closed subspace
$M$ of $H_{F}^{2}$ is $S_{F}$-almost invariant if and only if $M=N(T^{\ast
})=\left[  R(T)\right]  ^{\perp},$ where $T=T_{\Phi}\left(  I_{E_{1}%
}-T_{\Theta}T_{\Theta}^{\ast}\right)  .$ In this case, $\varsigma(S_{F},$
$N(T^{\ast}))=\varsigma(S_{F}^{\ast},R(T)).$
\end{corollary}

Since $T=T_{\Phi}\left(  I_{E_{1}}-T_{\Theta}T_{\Theta}^{\ast}\right)  $ is a
partial isometry,
\begin{equation}
I_{F}-TT^{\ast}=I_{F}-T_{\Phi}(I_{E_{1}}-T_{\Theta}T_{\Theta}^{\ast})T_{\Phi
}^{\ast} \label{null}%
\end{equation}
is the projection onto $N(T^{\ast}).$

\begin{corollary}
Under the same notation as in Theorem \ref{main}, a closed subspace
$M$ of $H_{F}^{2}$ is $S_{F}^{\ast}$-almost invariant if and only if either
$M=R(T_{\Theta})$ or $M$ is a reproducing kernel space whose kernel
$K_{M}(z,w)$ is of the form%
\begin{equation}
K_{M}(z,w)=\frac{\Phi(z)(I_{E_{1}}-\Theta(z)\Theta(w)^{\ast})\Phi(w)^{\ast}%
}{1-z\overline{w}}. \label{nullka}%
\end{equation}
Similarly, a closed subspace $M$ of $H_{F}^{2}$ is $S_{F}$-almost invariant if
and only if either $M=K_{\Theta}$ or $M$ is a reproducing kernel space whose
kernel is of the form%
\begin{equation}
K_{M^{\perp}}(z,w)=\frac{I_{F}-\Phi(z)(I_{E_{1}}-\Theta(z)\Theta(w)^{\ast
})\Phi(w)^{\ast}}{1-z\overline{w}}. \label{nullk}%
\end{equation}

\end{corollary}

\begin{proof}
Set $T:=T_{\Phi}\left(  I_{E_{1}}-T_{\Theta}T_{\Theta}^{\ast}\right)  .$ Since
$T$ is a partial isometry \cite{Fricain}, the reproducing kernel of $R(T)$ is
$TT^{\ast}\left(  k_{w}(z)I_{F}\right)  $, where $k_{w}(z)=1/(1-z\overline
{w})$ is the reproducing kernel of $H^{2}$ \cite{Fricain}. It follows from a
general fact that%
\begin{align*}
TT^{\ast}\left(  k_{w}(z)I_{F}\right)   &  =T_{\Phi}(I_{E_{1}}-T_{\Theta
}T_{\Theta}^{\ast})T_{\Phi}^{\ast}k_{w}(z)I_{F}\\
&  =\frac{\Phi(z)(I_{E}-\Theta(z)\Theta(w)^{\ast})\Phi(w)^{\ast}}%
{1-z\overline{w}}.
\end{align*}
Similarly, by (\ref{null}), the reproducing kernel of $N(T^{\ast})$ is given
by (\ref{nullk}).
\end{proof}

\

We note that the reproducing kernels of $R(T_{\Theta})$ and $K_{\Theta}$ in the
above corollary are respectively ($\Theta\in H_{B(E,F)}^{\infty}$ is inner)
\[
\frac{\Theta(z)\Theta(w)^{\ast}}{1-z\overline{w}}\text{ and }\frac
{I_{F}-\Theta(z)\Theta(w)^{\ast}}{1-z\overline{w}}.
\]
By Corollary \ref{equivalent}, the kernel in (\ref{nullk}) can be represented
as a kernel in (\ref{nullka}) with different $\Phi(z)$ and $\Theta(z),$ it
will be interesting to have a direct proof of this fact.


\section{Invariant subspaces of a finite rank perturbation of the shift
operator}

It has been observed in Proposition 1.3 in \cite{APTT} that an
almost invariant subspace of $T$ on a Banach space is actually an invariant
subspace of $T+T_{0}$, where $T_{0}$ is a finite rank operator.

\begin{lemma}
\textrm{\cite{APTT}} \label{apttl} Let $X$ be a Banach space, $T\in B(X)$ and
$M$ be a closed subspace of $X.$ Then $M$ is $T$-almost invariant
if and if $M$ is $(T+T_{0})$-invariant for some finite rank operator $T_{0}.$
\end{lemma}

It turns out that if $X$ is a Hilbert space, for a given $M,$ we can write
down all those $T_{0}$ such that $M$ is $(T+T_{0})$-invariant.
Then $M$ is $T$-invariant for $T\in
B(H)$ if and only if $TP_{M}-P_{M}TP_{M}=0$ and $M$ is $T$-reducing if and
only if $TP_{M}-TP_{M}=0.$ This inspires another equivalent notion of almost
invariance which appeared in literature much earlier, see for example
\cite{Hoffman}. To distinguish we temporarily give it a different term.

\begin{definition}
\label{def2}\textrm{\cite{Hoffman}} Let $T\in B(H)$ and $M$ be a closed
subspace of $H.$ Then we say that $M$ is $T$-finite rank
invariant if $TP_{M}-P_{M}TP_{M}$ is of finite rank and that $M$
is $T$-finite rank reducing if $TP_{M}-P_{M}T$ is of finite rank.
\end{definition}

It is known that the notion of $T$-finite rank invariant is equivalent to\ the
notion of $T$-almost invariant. We write down a proof to illustrate some points.

\begin{lemma}
Let $T\in B(H)$ and $M$ be a closed subspace of $H.$ Then $M$ is
$T$-finite rank invariant if and only if $M$ is $T$-almost invariant.
Similarly, $M$ is $T$-finite rank reducing if and only if $M$ is $T$-almost reducing.
\end{lemma}

\begin{proof}
Assume $W:=TP_{M}-P_{M}TP_{M}=(I-P_{M})TP_{M}$ is of finite rank. Then for any
$h\in M,$ $Th=P_{M}Th+Wh\in M+R(W).$ That is, $TM\subset M+R(W).$ So $R(W)$ is
a finite dimensional defect space of $M$ and $M$ is $T$-almost invariant. In
fact, $R(W)\perp M$ since $P_{M}W=P_{M}(I-P_{M})TP_{M}=0$. So $R(W)$ is the
minimal\ orthogonal defect space of $M.$

On the other hand, assume $M$ is $T$-almost invariant. That is, $TM\subset
M\oplus G,$ where $G$ is finite dimensional. Set $W:=TP_{M}-P_{M}TP_{M}.$ It
is clear that $W|M^{\perp}=0.$ For $h\in M,$ $Wh=Th-P_{M}Th=(I-P_{M})Th\in
M\oplus G.$ Thus $Wh\in G$ since $Wh\perp M.$ This proves $R(W)\subset G$ and
$W$ is a finite rank operator. Hence $M$ is $T$-finite rank invariant.

Similarly, assume $W:=TP_{M}-P_{M}T$ is of finite rank. Then, $TM\subset
M+R(W)$ and $T^{\ast}M\subset M+R(W^{\ast}).$ Hence $M$ is
$T$-almost reducing.

On the other hand, assume $M$ is $T$-almost reducing. By what we have just
proved, $W_{1}:=(I-P_{M})TP_{M}$ and $W_{2}:=(I-P_{M})T^{\ast}P_{M}$ are both
of finite rank. Then%
\[
TP_{M}-P_{M}T=W_{1}-W_{2}^{\ast}%
\]
is of finite rank. Hence $M$ is $T$-finite rank reducing.
\end{proof}

\begin{theorem}
\label{finiterank}Let $T\in B(H)$ and $M$ be a closed subspace of $H.$ Assume
$M$ is $T$-almost invariant. Then $M$ is $(T+T_{0})$-invariant for some finite
rank operator $T_{0}$ if and only if
\begin{equation}
T_{0}=-(I-P_{M})TP_{M}+%
{\textstyle\sum_{i=1}^{k}}
x_{i}\otimes y_{i}+%
{\textstyle\sum_{j=1}^{m}}
u_{j}\otimes v_{j} ,\label{t01}%
\end{equation}
where $x_{i}\in M$ and $y_{i}\in H$, $u_{j}\in H$ and $v_{j}\in M^{\perp}$ are arbitrary.

Similarly, assume $M$ is $T$-almost reducing. Then $M$ is $(T+T_{0})$-reducing
for some finite rank operator $T_{0}$ if and only if
\[
T_{0}=-(I-P_{M})TP_{M}-P_{M}T(I-P_{M})+%
{\textstyle\sum_{i=1}^{k}}
x_{i}\otimes y_{i}+%
{\textstyle\sum_{j=1}^{m}}
u_{j}\otimes v_{j}%
\]
where $x_{i},y_{i}\in M$ and $u_{j},v_{j}\in M^{\perp}$ are arbitrary.
\end{theorem}

\begin{proof}
Set $W:=(I-P_{M})TP_{M},$ $W_{1}=%
{\textstyle\sum_{i=1}^{k}}
x_{i}\otimes y_{i}$ and $W_{2}=%
{\textstyle\sum_{j=1}^{m}}
u_{j}\otimes v_{j}.$ Assume $T_{0}=-W+W_{1}+W_{2}.$ Then for $h\in M,$%
\begin{align*}
\left(  T+T_{0}\right)  h  &  =Th-\left(  TP_{M}-P_{M}TP_{M}\right)  h+%
{\textstyle\sum_{i=1}^{k}}
\left\langle h,y_{i}\right\rangle x_{i}\\
&  =P_{M}TP_{M}h+%
{\textstyle\sum_{i=1}^{k}}
\left\langle h,y_{i}\right\rangle x_{i}\in M.
\end{align*}
This proves ``if" direction.

Now assume $T_{0}$ is a finite rank operator such that $M$ is $(T+T_{0}%
)$-invariant. Then for $h\in M,$
\[
\left(  T_{0}+W\right)  h=T_{0}h+\left(  TP_{M}-P_{M}TP_{M}\right)  h=\left(
T+T_{0}\right)  h-P_{M}TP_{M}h\in M.
\]
Hence $\left(  T_{0}+W\right)  P_{M}$ is a finite rank operator and $R\left[
\left(  T_{0}+W\right)  P_{M}\right]  \subset M.$ Therefore, $\left(
T_{0}+W\right)  P_{M}=W_{1}$ for some $W_{1}$ of the form $%
{\textstyle\sum_{i=1}^{k}}
x_{i}\otimes y_{i}$. Note that
\[
\left(  T_{0}+W\right)  (I-P_{M})=T_{0}(I-P_{M})=W_{2}=%
{\textstyle\sum_{j=1}^{m}}
u_{j}\otimes v_{j}
\]
for some $W_{2}$ of the form $%
{\textstyle\sum_{j=1}^{m}}
u_{j}\otimes v_{j}.$ Thus%
\[
\left(  T_{0}+W\right)  =\left(  T_{0}+W\right)  P_{M}+T_{0}(I-P_{M}%
)=W_{1}+W_{2}.
\]
This proves (\ref{t01}).

Next we prove the reducing case. The ``if" direction comes from a
direct verification. We prove the ``only if" direction. Assume $T_{0}$ is a
finite rank operator such that $M$ is $(T+T_{0})$-reducing. Set
\[
Q:=T_{0}+(I-P_{M})TP_{M}+P_{M}T(I-P_{M}).
\]
Since $M$ is $T$-almost reducing, both $(I-P_{M})TP_{M}$ and $P_{M}T(I-P_{M})$
are of finite rank. Then for $h\in M,$%
\[
Qh=(T_{0}+T)h-P_{M}Th\in M,
\]
and for $h_{1}\in M^{\perp}$%
\[
Qh_{1}=T_{0}h_{1}+P_{M}Th_{1}=(T_{0}+T)h_{1}-(I-P_{M})Th_{1}\in M^{\perp}.
\]
Thus $QM\subset M$ and $QM^{\perp}\subset M^{\perp}.$ Consequently,
\[
Q=P_{M}QP_{M}+(I-P_{M})Q(I-P_{M})=%
{\textstyle\sum_{i=1}^{k}}
x_{i}\otimes y_{i}+%
{\textstyle\sum_{j=1}^{m}}
u_{j}\otimes v_{j}%
\]
where $x_{i},y_{i}\in M$ and $u_{j},v_{j}\in M^{\perp}.$ The proof is complete.
\end{proof}

\

For essentially invariant subspaces, its first appearance was in Brown and
Pearcy \cite{BrownPearcy} in 1971 where it was proved that every operator on a
complex infinite dimensional Hilbert space admits an essentially invariant
subspace. Here is the definition used in \cite{BrownPearcy} \cite{Hoffman}.

\begin{definition}
\label{defBP}\textrm{\cite{BrownPearcy} \cite{Hoffman}} Let $T\in B(H)$ and
$M$ be a closed subspace of $H.$ Then we say that $M$ is $T$-BP
essentially invariant if $TP_{M}-P_{M}TP_{M}$ is a compact operator
and that $M$ is $T$-BP essentially reducing if $TP_{M}-P_{M}T$ is
a compact operator.
\end{definition}

Inspired by Lemma \ref{apttl}, Sirotkin and Wallis \cite{Sirotkin2} gave the
following definition on a Banach space, but we state it on a Hilbert space.

\begin{definition}
\label{defSW}\textrm{\cite{Sirotkin2}}  Let $T\in B(H)$ and $M$ be a closed
subspace of $H.$ Then we say that $M$ is $T$-SW essentially
invariant if $M$ is $(T+T_{0})$-invariant for some compact operator $T_{0}$
and that  $M$ is $T$-SW essentially reducing if $M$ is
$(T+T_{0})$-reducing for some compact operator $T_{0}.$
\end{definition}

Inspired by Definition \ref{def1}, we can make the following definition.

\begin{definition}
\label{defG}\textrm{\cite{Sirotkin2}} Let $T\in B(H)$ and $M$ be a closed
subspace of $H.$ Then we say that $M$ is $T$-essentially
invariant if there exists a compact operator $G$ such that $TM\subset M+R(G)$
and that  $M$ is $T$-essentially reducing if both $M$ and
$M^{\perp}$ are $T$-essentially invariant.
\end{definition}

As expected, these three definitions are equivalent and we give a proof for clarity.

\begin{lemma}
\label{lem58} Let $T\in B(H)$ and $M$ be a closed subspace of $H.$
Then $M$ is $T$-BP essentially invariant if and only if $M$ is
$T$-SW essentially invariant if and only if $T$-essentially invariant. A
similar statement holds for essentially reducing subspaces.
\end{lemma}

\begin{proof}
Assume that $M$ is $T$-BP essentially invariant. Then $(I-P_{M})TP_{M}$ is compact
and $M$ is $\left[  T-(I-P_{M})TP_{M}\right]  $-invariant. Thus $M$ is $T$-SW
essentially invariant.

Assume $M$ is $T$-SW essentially invariant. That is $M$ is $\left[
T+T_{0}\right]  $-invariant for some compact operator $T_{0}.$ Thus $\left[
T+T_{0}\right]  M\subset M.$ Then it is easy to see that $TM\subset
M+R(T_{0}).$ Thus $M$ is $T$-essentially invariant.

Assume $M$ is $T$-essentially invariant. Let $G$ be a compact
operator such that $TM\subset M+R(G).$ Then%
\[
TM\subset M+R(G)=M+R(P_{M}G+(I-P_{M})G)\subset M\oplus R((I-P_{M})G).
\]
Set $W:=(I-P_{M})TP_{M}.$ Then $W|M^{\perp}=0$ and $R(W)\subset M^{\perp}.$
For $h\in M,$%
\[
Wh=Th-P_{M}Th\subset M\oplus R((I-P_{M})G).
\]
Therefore, $R(W)\subset R((I-P_{M})G)$. Since $(I-P_{M})G$ is compact, $W$ is
compact. This proves $M$ is $T$-BP essentially invariant. The proof for the
reducing case is similar.
\end{proof}

\begin{theorem}
Let $T\in B(H)$ and $M$ be a closed subspace of $H.$ Assume $M$ is
$T$-essentially invariant. Then $M$ is $(T+T_{0})$-invariant for some compact
operator $T_{0}$ if and only if
\begin{equation}
T_{0}=-(I-P_{M})TP_{M}+P_{M}W_{1}+W_{2}(I-P_{M}) ,\label{t03}%
\end{equation}
where $W_{1},W_{2}\in B(H)$ are two arbitrary compact operators.

Similarly, assume $M$ is $T$-essentially reducing. Then $M$ is $(T+T_{0}%
)$-reducing for compact operator $T_{0}$ if and only if
\[
T_{0}=-(I-P_{M})TP_{M}-P_{M}T(I-P_{M})+P_{M}W_{1}P_{M}+(I-P_{M})W_{2}%
(I-P_{M}),
\]
where $W_{1},W_{2}\in B(H)$ are two arbitrary compact operators.
\end{theorem}

\begin{proof}
Set $W:=(I-P_{M})TP_{M}.$ By assumption and Lemma \ref{lem58},
$W$ is a compact operator. Assume $T_{0}=-W+P_{M}W_{1}+W_{2}(I-P_{M}),$ where
$W_{1},W_{2}\in B(H)$ are compact. Then for $h\in M,$%
\begin{align*}
\left(  T+T_{0}\right)  h  &  =Th-\left(  TP_{M}-P_{M}TP_{M}\right)
h+P_{M}W_{1}h\\
&  =P_{M}TP_{M}h+P_{M}W_{1}h\in M.
\end{align*}
This proves ``if" direction.

Now assume $T_{0}$ is a compact operator such that $M$ is
$(T+T_{0})$-invariant. Then for $h\in M$,
\[
\left(  T_{0}+W\right)  h=T_{0}h+\left(  TP_{M}-P_{M}TP_{M}\right)  h=\left(
T+T_{0}\right)  h-P_{M}TP_{M}h\in M.
\]
Hence, $P_{M}\left(  T_{0}+W\right)  P_{M}=\left(  T_{0}+W\right)  P_{M}.$
Thus
\begin{align*}
T_{0}+W  &  =\left(  T_{0}+W\right)  P_{M}+\left(  T_{0}+W\right)  (I-P_{M})\\
&  =P_{M}\left(  T_{0}+W\right)  P_{M}+T_{0}(I-P_{M}).
\end{align*}
This proves (\ref{t03}) with $W_{1}=\left(  T_{0}+W\right)  P_{M}$ and
$W_{2}=T_{0}.$

The proof for the reducing case is similar.
\end{proof}

\

By using Theorem \ref{finiterank}, for a $S^{\ast}$-almost invariant subspace
$M$, we can write down all finite rank operators $T_{0}$ such that $M$ is
$(S+T_{0})$-invariant.

\begin{theorem}
Let $M:=R(T_{\Phi}\left(  I_{E_{1}}-T_{\Theta}T_{\Theta}^{\ast}\right)  )$ be
as in Theorem \ref{main}. The following statements hold.

\begin{itemize}
\item[\textrm{(i)}] $M$ is $(S_{F}+T_{0})$-invariant for some finite rank
operator $T_{0}$ if and only if
\begin{equation}
T_{0}=-T_{\Phi}T_{\Theta}P_{E}T_{\Theta}^{\ast}T_{\Phi}^{\ast}+%
{\textstyle\sum_{i=1}^{k}}
x_{i}\otimes y_{i}+%
{\textstyle\sum_{j=1}^{m}}
u_{j}\otimes v_{j}, \label{t0a}%
\end{equation}
where $x_{i}\in M$ and $y_{i}\in H$, $u_{j}\in H$ and $v_{j}\in M^{\perp}$ are arbitrary.

\item[\textrm{(ii)}] $M$ is $(S_{F}^{\ast}+T_{1})$-invariant for some finite
rank operator $T_{1}$ if and only if%
\[
T_{1}=-S_{F}^{\ast}T_{\Phi}P_{E_{1}}\left(  I-T_{\Theta}T_{\Theta}^{\ast
}\right)  T_{\Phi}^{\ast}+%
{\textstyle\sum_{i=1}^{k}}
x_{i}\otimes y_{i}+%
{\textstyle\sum_{j=1}^{m}}
u_{j}\otimes v_{j},
\]
where $x_{i}\in M$ and $y_{i}\in H$, $u_{j}\in H$ and $v_{j}\in M^{\perp}$ are arbitrary.

\item[\textrm{(iii)}] $M$ is $(S_{F}+T_{2})$-reducing for some finite rank
operator $T_{2}$ if and only if%
\[
T_{2}=-T_{\Phi}T_{\Theta}P_{E}T_{\Theta}^{\ast}T_{\Phi}^{\ast}-S_{F}^{\ast
}T_{\Phi}P_{E_{1}}\left(  I-T_{\Theta}T_{\Theta}^{\ast}\right)  T_{\Phi}%
^{\ast}+%
{\textstyle\sum_{i=1}^{k}}
x_{i}\otimes y_{i}+%
{\textstyle\sum_{j=1}^{m}}
u_{j}\otimes v_{j},
\]
where $x_{i},y_{i}\in M$ and $u_{j},v_{j}\in M^{\perp}$ are arbitrary.
\end{itemize}
\end{theorem}

\begin{proof}
Set $K:=T_{\Phi}\left(  I-T_{\Theta}T_{\Theta}^{\ast}\right) $
and $W:=T_{\Phi}T_{\Theta}P_{E}T_{\Theta}^{\ast}S_{E_{1}}.$ Then
(\ref{defGGG}) become $S_{F}K=KS_{E_{1}}+W.$ Since $K$ is a partial isometry,
$P_{M}=KK^{\ast}.$ Thus
\begin{align*}
(I-P_{M})S_{F}P_{M}  &  =(I-KK^{\ast})S_{F}KK^{\ast}\\
&  =(I-KK^{\ast})(KS_{E_{1}}+W)K^{\ast}\\
&  =(I-KK^{\ast})WK^{\ast}=WK^{\ast}-P_{M}WK^{\ast},
\end{align*}
where $(I-KK^{\ast})K=0.$ Since $\Theta$ is inner, we have
\begin{align*}
WK^{\ast}  &  =T_{\Phi}T_{\Theta}P_{E}T_{\Theta}^{\ast}S_{E_{1}}\left(
I-T_{\Theta}T_{\Theta}^{\ast}\right)  T_{\Phi}^{\ast}\\
&  =T_{\Phi}T_{\Theta}P_{E}T_{\Theta}^{\ast}\left(  \left(  I-T_{\Theta
}T_{\Theta}^{\ast}\right)  T_{\Phi}^{\ast}S-T_{\Theta}P_{E}T_{\Theta}^{\ast
}T_{\Phi}^{\ast}\right) \\
&  =T_{\Phi}T_{\Theta}P_{E}T_{\Theta}^{\ast}T_{\Phi}^{\ast}.
\end{align*}
Set $W_{1}:=%
{\textstyle\sum_{i=1}^{k}}
x_{i}\otimes y_{i}$ and $W_{2}:=%
{\textstyle\sum_{j=1}^{m}}
u_{j}\otimes v_{j},$ where $x_{i}\in M$, $y_{i},u_{j}\in H$, $v_{j}\in
M^{\perp}.$ By Theorem \ref{finiterank},
\[
T_{0}=-\left(  WK^{\ast}-P_{M}WK^{\ast}\right)  +W_{1}+W_{2}.
\]
Note that $P_{M}WK^{\ast}$ is a finite rank operator of the same form as
$W_{1}.$ Thus $T_{0}=-WK^{\ast}+W_{1}+W_{2}.$

Next we prove (ii). Set $G:=-T_{\Phi}S_{E_{1}}^{\ast}T_{\Theta}P_{E}T_{\Theta
}^{\ast}+S_{F}^{\ast}T_{\Phi}P_{E_{1}}\left(  I-T_{\Theta}T_{\Theta}^{\ast
}\right) .$ Then (\ref{add1}) becomes $S_{F}^{\ast}K=KS_{E_{1}%
}^{\ast}+G.$ Note that%
\begin{align}
(I-P_{M})S_{F}^{\ast}P_{M}  &  =(I-KK^{\ast})S_{F}^{\ast}KK^{\ast}\nonumber\\
&  =(I-KK^{\ast})\left(  KS_{E_{1}}^{\ast}+G\right)  K^{\ast}\nonumber\\
&  =(I-KK^{\ast})GK^{\ast}=GK^{\ast}-P_{M}GK^{\ast}. \label{starf}%
\end{align}
Since $\Theta$ is inner,%
\begin{align*}
GK^{\ast}  &  =\left[  -T_{\Phi}S_{E_{1}}^{\ast}T_{\Theta}P_{E}T_{\Theta
}^{\ast}+S_{F}^{\ast}T_{\Phi}P_{E_{1}}\left(  I-T_{\Theta}T_{\Theta}^{\ast
}\right)  \right]  \left(  I-T_{\Theta}T_{\Theta}^{\ast}\right)  T_{\Phi
}^{\ast}\\
&  =S_{F}^{\ast}T_{\Phi}P_{E_{1}}\left(  I-T_{\Theta}T_{\Theta}^{\ast}\right)
T_{\Phi}^{\ast}.
\end{align*}
By Theorem \ref{finiterank},
\[
T_{1}=-\left(  GK^{\ast}-P_{M}GK^{\ast}\right)  +W_{1}+W_{2}.
\]
Note that $P_{M}GK^{\ast}$ is a finite rank operator of the same form as
$W_{1}.$ Thus $T_{1}=-GK^{\ast}+W_{1}+W_{2}.$

Now we prove (iii). By taking adjoint, (\ref{starf}) becomes
\[
P_{M}S_{F}(I-P_{M})=KG^{\ast}-KG^{\ast}P_{M}.
\]
Set $W_{1}:=%
{\textstyle\sum_{i=1}^{k}}
x_{i}\otimes y_{i}$ and $W_{2}:=%
{\textstyle\sum_{j=1}^{m}}
u_{j}\otimes v_{j},$ where $x_{i},y_{i}\in M$, $u_{j},v_{j}\in M^{\perp}.$ By
Theorem \ref{finiterank},
\[
T_{2}=-\left(  WK^{\ast}-P_{M}WK^{\ast}\right)  -\left(  GK^{\ast}%
-P_{M}GK^{\ast}\right)  +W_{1}+W_{2}.
\]
Since $M=R(K),$ $P_{M}K=K$ and $K^{\ast}=K^{\ast}P_{M}.$ Thus $P_{M}WK^{\ast
}=P_{M}WK^{\ast}P_{M}$ and $P_{M}GK^{\ast}=P_{M}WK^{\ast}P_{M}.$ So
$P_{M}WK^{\ast}P_{M}$ and $P_{M}WK^{\ast}P_{M}$ are finite rank operators of
the same form as $W_{1}.$ Therefore, $T_{2}=-WK^{\ast}-GK^{\ast}+W_{1}+W_{2}.$
\end{proof}

\

We end up the paper with a question. Let $G$ be a given finite
rank operator on $H_{F}^{2}.$ If $M$ is $(S_{F}+G)$-invariant, then $M$ is
$S_{F}$-almost invariant and a subspace of $R(G)$ is a defect space. Thus
$M=R(T_{\Phi}\left(  I-T_{\Theta}T_{\Theta}^{\ast}\right)  )$. In
this case, the following question naturally arises:

\begin{question}
How do we find these $\Phi$ and $\Theta$ in terms of $G?$
\end{question}

\

\noindent \textit{Acknowledgments}. \ The authors are deeply
indebted to the referee of a previous version for many suggestions
on the structure, substance, and style of the paper, which have
helped improved the presentation. \ The work of the second named
author was supported by NRF(Korea) grant No. 2022R1A2C1010830. \ The
work of the fourth named author was supported by NRF(Korea) grant
No. 2021R1A2C1005428.

\


\

Caixing Gu

Department of Mathematics, California Polytechnic State University, San Luis
Obispo, CA 93407, USA

E-mail: cgu@calpoly.edu

\bigskip

In Sung Hwang

Department of Mathematics, Sungkyunkwan University, Suwon 16419, South Korea

E-mail: ihwang@skku.edu

\bigskip

Hyoung Joon Kim

Department of Mathematics and RIM, Seoul National University, Seoul 08826, Korea

E-mail: hjkim76@snu.ac.kr

\bigskip

Woo Young Lee

HCMC, Korea Institute for Advanced Study (KIAS), Seoul, 02455, Korea

E-mail: wylee@snu.ac.kr

\bigskip

Jaehui Park

Department of Mathematics Education, Chonnam National University,
Gwangju, 61186, Korea

E-mail: hiems1855@gmail.com


\begin{thebibliography}{99}                                                                                               %


\bibitem {APTT}G. Androulakis, A.I. Popov, A. Tcaciuc, and V.G. Troitsky,
Almost invariant half-spaces of operators on Banach spaces,
\textit{Integral Equations Operator Theory} \textbf{65} (2009),
no. 4, 473--484.

\bibitem {Benhida}C. Benhida and D. Timotin, Finite rank perturbations of
contractions, \textit{Integral Equations Operator Theory}
\textbf{36} (2000), no. 3, 253-268.

\bibitem {Beurling}A. Beurling, On two problems concerning linear
transformations in Hilbert space. \textit{Acta Math.} \textbf{81} (1949), 239--255 .

\bibitem {BS2}A. B\"{o}ttcher and B. Silbermann, \textit{Analysis of Toeplitz
Operators}, 2nd edition, Springer-Verlag, 2006.

\bibitem {BrownDouglas}A. Brown and R. G. Douglas, Partially isometric
Toeplitz operators, \textit{Proc. Amer. Math. Soc.}
\textbf{16} (1965) 681-682.

\bibitem {BrownPearcy}A. Brown and C. Pearcy, Compact restrictions of
operators, \textit{Acta Sci. Math. (Szeged)}
\textbf{32} (1971) 271--282.

\bibitem {CRoss}M.C. C\^{a}mara and T. Ross, The dual of the compressed shift.
\textit{Canad. Math. Bull.}
\textbf{64} (2021), no. 1, 98-111.

\bibitem {ChDas}A. Chattopadhyay, S. Das, and C. Pradhan, Almost invariant
subspaces of the shift operator on vector-valued Hardy spaces,
\textit{Integral Equations Operator Theory}
\textbf{92} (2020), no. 6, Paper No. 49.

\bibitem {ChalendarCP}I. Chalendar, N. Chevrot and J.R. Partington, Nearly
invariant subspaces for backwards shifts on vector-valued Hardy spaces,
\textit{J. Operator Theory} \textbf{63} (2010), no. 2, 403-415.

\bibitem {ChalendarGP}I. Chalendar, E.A. Gallardo-Gutirrez, and J.R.
Partington, A Beurling Theorem for almost-invariant subspaces of the shift
operator,
\textit{J. Operator Theory} \textbf{83} (2020), no. 2, 321-331.

\bibitem {Chevrot}N. Chevrot, Kernel of vector-valued Toeplitz operators,
\textit{Integral Equations Operator Theory}
\textbf{67} (2010), no. 1, 57-78.

\bibitem {CHL}R.E. Curto, I. S. Hwang and W. Y. Lee, The Beurling--Lax--Halmos
theorem for infinite multiplicity,
\textit{J. Funct. Anal.} \textbf{280} (2021), no. 6, Paper No.
108884, 101 pp.

\bibitem {Fricain}E. Fricain and J. Mashreghi,
\textit{The theory of $H(b)$ spaces}.
Vol. 1 and Vol. 2, New Mathematical Monographs, 20 and 21 (2016), Cambridge
University Press, Cambridge.

\bibitem {GuH}C. Gu, Multiplcative properties of infinite block Toeplitz and
Hankel matrices, Toeplitz Operators and Random Matrices (In Memory of Harold
Widom), 371-399, Oper. Theory Adv. Appl. 289, Birkhauser, 2022.

\bibitem {GuDOParkj}C. Gu, D. Kang and J. Park, An inverse problem for kernels
of block Hankel operators,
\textit{Bull. Sci. Math.} \textbf{171} (2021) Paper No. 103021, 23 pp.

\bibitem {GuLuo}C. Gu and S. Luo, Invariant subspaces of the direct sum of
forward and backward shifts on vector-valued Hardy spaces.
\textit{J. Funct. Anal.} \textbf{282}
(2022), no. 9, Art. No. 109419, 31 pp.

\bibitem {Jung}I. Jung, E. Ko, C. Pearcy, Almost invariant half-spaces for
operators on Hilbert space,
\textit{Bull. Aust. Math. Soc.} \textbf{97} (2018), no.1, 133-140

\bibitem {Halmos}P. Halmos, Shifts on Hilbert spaces. J. Reine Angew. Math.
208 (1961), 102--112.

\bibitem {Hitt}D. Hitt, Invariant subspaces of $H^{2}$ of an annulus,
\textit{Pacific J. Math.} \textbf{134} (1988), no. 1, 101-120.

\bibitem {Hoffman}M.J. Hoffman, Spans and intersections of essentially
reducing subspaces,
\textit{Proc. Amer. Math. Soc.} \textbf{72} (1978), no.2, 333--340.

\bibitem {Lax}P. Lax, Shift invariant spaces.
\textit{Acta Math.} \textbf{101} (1959), 163--178.

\bibitem {LR15}S. Luo and S. Richter, Hankel operators and invariant subspaces
of the Dirichlet space.
\textit{J. Lond. Math. Soc. (2)} \textbf{91} (2015), no. 2, 423--438.

\bibitem {OLoughlin}R. O'Loughlin, Nearly invariant subspaces with
applications to truncated Toeplitz operators,
\textit{Complex Analysis and Operator
Theory} \textbf{14} (2020), no. 8, Paper No. 86, 24 pages.

\bibitem {Pe}V.V. Peller, \textit{Hankel Operators and Their Applications},
Springer-Verlag, 2003

\bibitem {PopovT}A. Popov and A. Tcaciuc, Every operator has almost-invariant
subspaces,
\textit{J. Funct. Anal.} \textbf{265} (2013), no. 2, 257--265.

\bibitem {Sarason}D. Sarason, Nearly invariant subspaces of the backward
shift, in
\textit{Contributions to Operator Theory and its Applications (Mesa, AZ, 1987)},
481-493, Oper. Theory Adv. Appl., 35, Birkh\" auser, Basel, 1988, pp. 481--493.

\bibitem {Sarason2}D. Sarason, Kernels of Toeplitz operators,
in
\textit{Toeplitz operators and related topics (Santa Cruz, CA, 1992)},
Oper. Theory Adv. Appl., 71, Birkh\" auser, Basel, 1994, pp. 153--164.

\bibitem {Sirotkin}G. Sirotkin and B. Wallis, The structure of
almost-invariant half-spaces for some operators,
\textit{J. Funct. Anal.} \textbf{267} (2014),
no. 7, 2298--2312.

\bibitem {Sirotkin2}G. Sirotkin and B. Wallis, Almost-invariant and
essentially-invariant halfspaces,
\textit{Linear Algebra Appl.} \textbf{507} (2016), 399-413.

\bibitem {NFBK}B. Sz.-Nagy, C. Foias, H. Bercovici and L. K\'{e}rchy,
\textit{Harmonic
Analysis of Operators on Hilbert Space}, Springer, 2010.

\bibitem {Tacacius}A. Tcaciuc, The invariant subspace problem for rank one
perturbations,
\textit{Duke Math. J.} \textbf{168} (2019), no. 8, 1539--1550.

\bibitem {Timotin}D. Timotin, The invariant subspaces of $S\oplus S^{\ast},$
\textit{Concrete. Operator} \textbf{7} (2020), 116--123.
\end{thebibliography}
\end{document}